%% file: main.tex
\documentclass[12pt]{amsart}
\input{preamble.tex}
\newcommand{\ls}{\otimes}
\newcommand{\gbinom}[3]{{{#1} \brack {#2}}_{#3}}
\usepackage{xcolor}

\begin{document}

\title[Geometric Erd\H{o}s-Hajnal]{Rainbow triangles and the Erd\H{o}s-Hajnal problem in projective geometries}
\author[Chun]{Carolyn Chun}
\address{University of Maryland, USA}
\email{cchun@umd.edu}
\author[Douthitt]{James Dylan Douthitt}
\address{Syracuse University, USA}
\email{jdouth5@lsu.edu}
\author[Ge]{Wayne Ge}
\address{Louisiana State University, USA}
\email{yge4@lsu.edu}
\author[Huynh]{Tony Huynh}
\address{IBS Discrete Mathematics Group (DIMAG), South Korea}
\email{tony.@ibs.re.kr}
\author[Kroeker]{Matthew E. Kroeker}
\address{Technische Universität Bergakademie Freiberg, Germany}
\email{mekroeker@uwaterloo.ca}
\author[Nelson]{Peter Nelson}
\address{University of Waterloo, Canada}
\email{apnelson@uwaterloo.ca}


\date{\today}
\begin{abstract}
	We formulate a geometric version of the \erdos-Hajnal conjecture that applies
  to finite projective geometries rather than graphs, in both its usual `induced' form and the
  multicoloured form. The multicoloured conjecture states, roughly,
  that a colouring $c$ of the points of $\PG(n-1,q)$ containing no copy of 
  a fixed colouring $c_0$ of $\PG(k-1,q)$ for small $k$ must contain a subspace
  of dimension polynomial in $n$ that avoids some colour. 
  
  If $(k,q) = (2,2)$, then $c_0$ is a colouring of a three-element `triangle', 
  and there are three essentially different cases, all of which we resolve. 
  We derive both the cases where $c_0$ assigns the same colour to two different elements 
  from a recent breakthrough result in additive combinatorics
  due to Kelley and Meka. We handle the case
  that $c_0$ is a `rainbow' colouring
  by proving that rainbow-triangle-free colourings of projective geometries
  are exactly those that admit a certain decomposition into two-coloured pieces.
  This is closely analogous to a theorem of Gallai on rainbow-triangle-free coloured complete graphs. 
    
  We also show that existing structure theorems resolve certain two-coloured cases 
  where $(k,q) = (2,3)$, and $(k,q) = (3,2)$. 
\end{abstract}

\maketitle
\sloppy
\section{Introduction}

The \erdos-Hajnal conjecture~[\ref{eh77}] states that for every graph $H$,
there is a constant $\beta = \beta(H) > 0$ such that, if $G$ is a graph on $n$ vertices
with no induced $H$-subgraph, then $G$ has a clique or stable set of size at least $n^\beta$. 
This problem is known for some particular graphs $H$ (notably the $5$-cycle) as well as 
when $H$ belongs to one of a few infinite classes such as the cographs,
but is wide open in general. See [\ref{chudnovsky}] for a survey. 

\erdos\ and Hajnal~[\ref{eh89}] in fact stated a multicoloured generalization of the same conjecture. 
To state this, we introduce some terminology. 
A \emph{colouring} of a finite set $X$ is any function $c$ whose domain is $X$, 
and, for $s \in \bN$, an \emph{$s$-colouring} is a function $c : X \to \{1, \dotsc, s\}$. 
Given a colouring $c$, we write $[c]$ for the range of $c$ (the set of colours used by $c$), 
and $|c|$ for the number of colours used by $c$.


For a graph $G$, a \emph{colouring of $G$} means a colouring of $E(G)$. 
The conjecture is stated in terms of colourings of complete graphs `containing' other colourings. 
If $c$ and $d$ are colourings of complete graphs $H$ and $G$ respectively,
then we say that $d$ \emph{contains} $c$ if there is an injection $\psi : V(H) \to V(G)$
such that $c(xy) = d(\psi(x)\psi(y))$ for all distinct $x,y \in V(H)$. 
A colouring not containing $c$ is \emph{$c$-free.} 

\begin{conjecture}[Multicoloured \erdos-Hajnal conjecture for graphs]\label{ehmcgraph}
  
  For all $1 \le s_0 \le s$ and $k \ge 1$, and every $s_0$-colouring $c_0$ of $K_k$, 
  there exist $\delta,C > 0$ such that, for all $n \in \bN$ 
  and every $c_0$-free $s$-colouring $c$ of $K_n$,
  there is a colour $i \in \{1, \dotsc, s\}$
  and a set $X$ of at least $C n^\delta$ vertices of $K_n$
  such that $c(e) \ne i$ for every edge $e$ with both ends in $X$.
\end{conjecture}

The set $X$ is referred to as \emph{homogeneous}. 
The case where $s_0 = s = 2$ is the classical \erdos-Hajnal 
conjecture. 

We make an analogous conjecture about colourings of projective geometries.
In this context, we write $G \cong \PG(n-1,q)$ to denote the projective 
space associated with some generic $n$-dimensional vector space $V$ over $\GF(q)$; 
the points of $G$ are the one-dimensional subspaces of $V$. 
Going forward, we often identify this object with the corresponding simple rank-$n$ 
$\GF(q)$-representable matroid on $\tfrac{q^n-1}{q-1}$ elements. 

A \emph{colouring} of $G$ is a colouring of the $\tfrac{q^n-1}{q-1}$ points of $G$. 
If $c$ and $d$ are colourings of $H \cong \PG(m-1,q)$ and $G \cong \PG(n-1,q)$
respectively, then we say that $d$ \emph{contains} $c$ if there is an injective
homomorphism $\psi : H \to G$ such that $c(x) = d(\psi(x))$ for every point $x$ of $H$. 
If $d$ does not contain $c$, then $d$ is \emph{$c$-free}. 
We are primarily interested in the binary case, where one can think of $n$-dimensional binary projective 
space more concretely as the set of nonzero vectors in $\bF_2^n$, and homomorphisms as linear maps. 

Given $s \in \bN$, and an $s$-colouring $c$ of $G \cong \PG(n-1,2)$, we say that 
a subspace $V$ of $G$ is \emph{$(c,s)$-homogeneous}
if there is some $i \in \{1, \dotsc, s\}$ 
for which $i \notin c(V)$. That is, the colouring $c$ avoids some colour on $V$.

\begin{conjecture}[Multicoloured geometric \erdos-Hajnal conjecture]\label{multicolouredeh}
  For every prime power $q$, all $k, s_0, s \in \bN$ with $s_0 \le s$, and every $s_0$-colouring $c_0$ of $\PG(k-1,q)$, 
  there exists $\delta,C > 0$ such that, 
  for all $n \in \bN$ and every $c_0$-free $s$-colouring $c$ of $G \cong \PG(n-1,q)$, 
  there is a $(c,s)$-homogeneous subspace of $G$ with dimension at least $C n^\delta$.
\end{conjecture}

\subsection*{Induced restrictions}

The classical \erdos-Hajnal conjecture on induced subgraphs is  
precisely the case of Conjecture~\ref{ehmcgraph} where both the colourings use only two colours.
Specializing Conjecture~\ref{multicolouredeh} in the same way gives the following conjecture, 
phrased in terms of \emph{induced restrictions}. 
In what follows, we slightly abuse notation by identifying a projective 
geometry $\PG(n-1,q)$ with its set of points. 
If $H$ and $X$ are sets of points in $\PG(k-1,q)$ and $\PG(n-1,q)$ respectively,
then we say that $X$ contains $H$ as an \emph{induced restriction}
if there is an injective homomorphism $\psi : \PG(k-1,q) \to \PG(n-1,q)$
such that, for each point $x$ of $\PG(k-1,q)$, 
we have $\psi(x) \in X$ if and only if $x \in H$.

\begin{conjecture}[Geometric \erdos-Hajnal conjecture]\label{inducedeh}
  For every prime power $q$ and all $k \in \bN$ and $H \subseteq \PG(k-1,q)$, there exists $\delta,C > 0$ such that, 
  for all $n \in \bN$ and $X \subseteq \PG(n-1,q)$ not containing $H$ as an induced restriction,
  there is a subspace $W$ of $\PG(n-1,q)$ such that $\dim(W) > C n^{\delta}$, 
  and either $W \subseteq X$ or $W \cap X = \varnothing$. 
\end{conjecture}

Finally, we state a strengthening of the last two conjectures, suggested by Jacob Fox, 
which asserts that in fact they hold for linear functions rather than polynomial. 
This may seem quite strong, but we know of no construction that refutes it. 

\begin{conjecture}
  Conjectures~\ref{multicolouredeh} and \ref{inducedeh} hold with $\delta = 1$. 
\end{conjecture}



\subsection*{Results} 
Our first few main results concern the binary case, where the object being excluded is
a colouring of the rank-two binary projective geometry $T \cong \PG(2-1,2)$, 
known in matroid parlance as a `triangle'.
This geometry $T$ has just three elements;
when thinking of a larger host geometry $G \cong \PG(n-1,2)$ 
as a set of vectors in $\bF_2^n \del \{0\}$, a copy of $T$ in $G$
is simply a triple of distinct vectors that sum to zero. 
The analogous problem for triangles in edge-coloured complete graphs is especially well-studied 
[\ref{asw}, \ref{fgp}, \ref{gallai}, \ref{gallaitrans}]. 

To prove Conjecture~\ref{multicolouredeh}, 
it suffices to consider the case where the colouring $c_0$ is surjective: 
that is, where $s_0 = |c_0|$. 
Thus, by symmetry, there are only three essentially different colourings $c_0$ of 
$T = \{x_1,x_2,x_3\}$: 
the $1$-colouring $c_0$ with $[c_0] = \{1\}$, 
the $2$-colouring $c_0$ where $c_0(x_1) = c_0(x_2) = 1$ and $c_0(x_3) = 2$, and 
the `rainbow' $3$-colouring with $c_0(x_i) = i$ for all $i$. 
Our first result handles the first two cases, for arbitrary $s \ge s_0$. 

\begin{theorem}\label{mainnonrainbow}
  Let $s_0 \in [2]$ and $s \ge s_0$, and $c_0$ be an $s_0$-colouring of $T$.
  Then Conjecture~\ref{multicolouredeh} holds for $c_0$ and $s$, with $\delta = \tfrac{1}{7}$. 
\end{theorem}

The proof of this theorem is via a reduction to a recent breakthrough 
result in additive combinatorics due to Kelley and Meka [\ref{km}], 
formulated by Hunter and Pohoata [\ref{hp}] in our setting.

Our next result handles the rainbow colouring of $T$, when $s_0 = s = 3$. 
\begin{theorem}\label{mainrainbow}
  Let $c_0$ be the rainbow $3$-colouring of $T$. 
  Then Conjecture~\ref{multicolouredeh} holds for $c_0$ and $s = 3$, 
  with $C = \tfrac{1}{3}$ and $\delta = 1$. 
  Moreover, the values of $\delta$ and $C$ are best-possible. 
\end{theorem}

The proof is via very different techniques. 
Following a route analogous to results due to Gallai [\ref{gallai}, \ref{gallaitrans}] 
and Fox, Grinshpun and Pach [\ref{fgp}],
we prove an exact structure theorem, 
Theorem~\ref{mainstructure}, which describes the $3$-colourings 
of $\PG(n-1,2)$ with no rainbow triangle, and derive Theorem~\ref{mainrainbow} 
as an easy consequence. 

Theorem~\ref{mainstructure} is a little too technical to state precisely here, 
but it roughly states that, if $c$ is a colouring of the points of a projective geometry
having no three colinear points of different colours, then $c$
can be constructed from two-coloured subspaces using a certain `join' operation. 
The following weakening will be enough to imply Theorem~\ref{mainrainbow}. 

\begin{theorem}\label{maintargetintro}
    Let $n \in \bN$, let $q$ be a prime power, 
    and $c$ be a colouring of the points of $G \cong \PG(n-1,q)$ so that $|c(L)| \le 2$
    for each two-dimensional subspace $L$ of $G$. 
    Then there are subspaces $\es = W_0 \subseteq W_1 \subseteq \dotsc \subseteq W_k = G$ of $G$
    such that $|c(W_{i+1} - W_i)| \le 2$ for all $0 \le i < k$. 
\end{theorem}

Note that this applies to all finite projective geometries, not just those over $\bF_2$; 
this is also the case for Theorem~\ref{mainstructure}. 
Sadly, for larger fields, Theorem~\ref{mainstructure} has no direct applications 
to Conjecture~\ref{multicolouredeh}, because the notion of a 
`line' diverges from that of a `triangle'. 

In the case where the number of colours is equal to the dimension, 
our structure theorem has an especially nice consequence, stating that there is
a unique $n$-colouring of $\PG(n-1,q)$ with no rainbow triangle, 
up to automorphisms and permutations of the colour set. 
In the case where $q = 2$, this follows from a much more general result 
of B\'{e}rczi and Schwarcz [\ref{bs}, Theorem 4], who call these colourings `standard'.
(Their result has an apparently stronger assumption that omits rainbow circuits rather than triangles, 
but we will see in Theorem~\ref{easyequiv} that this hypothesis is equivalent to ours in   
the case of projective geometries.)

\begin{theorem}\label{binaryintro}
    Let $q$ be a prime power, and $c$ be an $n$-colouring of $\PG(n-1,q)$ that uses all $n$ 
    colours, such that each line contains at most two colours. 
    Then there is a maximal chain of subspaces $\es = F_0 \subset \dotsc \subset F_n = \PG(n-1,q)$
    such that the sets $F_i - F_{i-1}$ are monochromatic with different colours. 
\end{theorem}

Conjecture~\ref{inducedeh} is equivalent to the case of Conjecture~\ref{multicolouredeh}
where $s = s_0 = 2$, so is resolved in the case where $q = k = 2$ 
by Theorem~\ref{mainnonrainbow}. We can resolve two more cases of 
Conjecture~\ref{mainnonrainbow} (up to complementation)
using known exact structure theorems. 

The first is in the case where $q = 3$ and $k = 2$, and $H$ is a set of 
two distinct elements in $\PG(1,3)$. Viewed as a two-colouring, this object $H$ is 
self-complementary, since $\PG(1,3)$ has four elements and symmetric automorphism group. 
Conjecture~\ref{inducedeh} will follow for $H$ from a characterization 
of the $H$-free matroids due to Mizell and Oxley [\ref{mo}],
which we prove separately and in more generality as 
Theorem~\ref{targetiffline}.

\begin{theorem}\label{mainchar3}
  Let $H \subseteq \PG(1,3)$ with $|H| = 2$. Then Conjecture~\ref{inducedeh} holds
  for $H$, with $\delta = 1$ and $C = \tfrac{1}{2}$. These values are best-possible. 
\end{theorem}

The second case we can settle is that of the \emph{claw}, which is defined as a
basis of the seven-point Fano plane $\PG(2,2)$. This will follow from the main theorem of [\ref{nn}], which allows us 
to reduce the problem to Theorem~\ref{mainnonrainbow}.

\begin{theorem}\label{mainclaw}
  Let $H \subseteq \PG(2,2)$ consist of either a (three-element) basis of $\PG(2,2)$
  or its four-element complement. Then Conjecture~\ref{inducedeh} holds for $H$, with $\delta = \tfrac{1}{7}$. 
\end{theorem}

\section{Preliminaries}

We write $\bN$ for the set $\{0, 1, \dotsc \}$, and $[n] = \{1, \dotsc, n\}$ for each $n \in \bN$. 
When discussing matroids, we usually use the notation and terminology of Oxley 
[\ref{oxley}]. 
For sets $X,Y$ in a matroid $M$, 
we write $\sqcap_M(X,Y)$ for the `local connectivity' evaluated as $r_M(X) + r_M(Y) - r_M(X \cup Y)$. 
A rank-one flat is a \emph{point}, and a rank-two flat is a \emph{line}. 
Two flats $F,F'$ of a matroid $M$ are \emph{complementary} 
  if $\sqcap_M(F,F') = 0$ and $F \cup F'$ is spanning in $M$;
such an $F'$ is a \emph{complement} of $F$, and satisfies $r(M) = r_M(F) + r_M(F')$. 
Note that if $B$ is any basis for $M$ containing a basis $I$ for the flat $F$,
then $\cl_M(B - I)$ is an example of a complement of $F$,  so every flat has at least one complement. 

Our first tool is due to Bose and Burton [\ref{bb}]. 
A \emph{projective subspace} of $\bF_2^n$ is a linear subspace with the zero vector removed. 

\begin{theorem}\label{bbthm}
  Let $t \in \bN$. 
  If $X \subseteq \bF_2^n \del \{0\}$ contains no $(t+1)$-dimensional projective subspace,
  then $|X| \le (1 - 2^{-t})2^n$. 
\end{theorem}

Given a set $X \subseteq \bF_2^n$, write $\omega(X)$ for the largest dimension of a projective
subspace contained in $X$, and $\alpha(X) = \omega (\bF_2^n \del X)$ 
for the largest dimension of a projective subspace contained in the complement of $X$. 
By applying Theorem~\ref{bbthm} to set $(\bF_2^n \del \{0\}) \del X$ with $t = \alpha(X)$, 
we obtain the following corollary. 

\begin{corollary}\label{bbomega}
  Each $X \subseteq \bF_2^n$ satisfies $|X| \ge 2^{n- \alpha(X)} - 1$. 
\end{corollary}


Finally, the qualitative Ramsey result that facilitates the
discussion ahead is due to Graham, Leeb and Rothschild [\ref{gr}].

\begin{theorem}\label{grthm}
  There is a function $R_q(s,t)$ so that, for each prime power $q$ and all $s,t \in \bN$,
  for all $n \in \bN$ with $n \ge R_q(s,t)$ and each colouring $c$ of $G \cong \PG(n-1,q)$ 
  with $[c] \subseteq \{\mathsf{Red}, \mathsf{Blue}\}$,
  there is either an $s$-dimensional subspace of $G$ coloured $\mathsf{Red}$,
  or a $t$-dimensional subspace of $G$ coloured $\mathsf{Blue}$. 
\end{theorem}

\section{Monochromatic Triangles}

In this section, we prove Theorem~\ref{mainnonrainbow}. 
We first review the relevant material in additive combinatorics, 
and state the main theorem of [\ref{hp}] in a slightly stronger form.
Logarithms here are all binary, and for $X,Y \subseteq \bF_2^n$, 
$X+ Y$ denotes the sum-set $\{x + y \colon x \in X, y \in Y\}$.

For $\delta \in (0,1]$, we follow [\ref{hp}] by writing $\cL(\delta) = \lceil \log (1/\delta) \rceil$ for the smallest integer $k$ such that $\delta 2^k \ge 1$; note that $k \le n$. 
For $X,V \subseteq \bF_2^n$, write $\mu_V(X) = \frac{|X \cap V|}{|V|}$ and $\mu(X) = \mu_{\bF_2^n}(X) = \frac{|X|}{2^n}$. 
The following is a convenient but slightly weaker 
rephrasing of Theorem~\ref{bbthm} in this language.

\begin{theorem}\label{bbthm2}
  If $\varnothing \ne X \subseteq \bF_2^n$, then $\alpha(X) + 1 \ge \cL(\mu(X))$. 
\end{theorem}
\begin{proof}
  We may assume that  $\alpha(X) < n$. 
  By Corollary~\ref{bbomega}, we have 
  \[ \cL(\mu(X)) = \ceil{\log\left(\tfrac{2^n}{|X|}\right)} = n - \floor{\log |X|} \le n - \floor{\log(2^{n-\alpha(X)}-1)}.\]
  Since $\alpha(X) < n$, we have $\floor{\log(2^{n-\alpha(X)} - 1)} = n-\alpha(X)-1$, and the result follows. 
\end{proof}

The following is the hard result we need from additive combinatorics. 
It follows from a specialized strengthening to $\bF_2$, due to Hunter and Pohoata,
of recent work of Kelley and Meka [\ref{km}]. 
The statement below is a weakening of [\ref{hp}, Proposition 3.1] in the case where $r = 1$. 

\begin{theorem}\label{hplem}
  Let $A \subseteq \bF_2^n$, and $\delta, \gamma \in (0,1)$. 
  There exists a subspace $V$ of $\bF_2^n$ with $\codim(V) = O(\cL(\delta)^5 \cL(\gamma)^2)$ such that either $\mu_V(A) < \delta$,
  or $\mu_V(A+A) > 1 - \gamma$. 
\end{theorem}

Using this, we can prove a key result showing that for every set $X \subseteq \bF_2^n$, 
either the sum-set $X + X$ or the complement of $X$ contains a subspace whose dimension 
is polynomial in $n$. The proof is a straightforward consequence of the techniques in [\ref{hp}]. 

\begin{theorem}\label{complorsum}
  There exists $C = C_{\ref{complorsum}} > 0$ such that every $X \subseteq \bF_2^n$ satisfies
   $\max(\alpha(X), \omega(X+X)) \ge C n^{1/7}$. 
\end{theorem}
\begin{proof}
  Let $C'$ be the constant implicit in Lemma~\ref{hplem}.
  We show that $C = \min(1 - 2^{-7}, \tfrac{1}{4}(C')^{-1/7})$ satisfies the theorem. 
  Let $X \subseteq \bF_2^n$ and let $d = \lceil C n ^{1/7} \rceil$. 
  If $Cn^{1/7} \le 1$, the result is trivial, so we may assume that $d \le 2Cn^{1/7}$. 

  Let $V$ be a subspace given by Lemma~\ref{hplem} with $\delta = \gamma  = 2^{-d}$,
  so $\cL(\delta) = \cL(\gamma) = d$ and thus, by the choice of $C$, we have 
  \[\dim(V) \ge n - C'd^7 \ge n - C' (2Cn^{1/7})^7 \ge n (1 - 2^{-7}) \ge Cn^{1/7}.\]

  By the choice of $\delta,\gamma$ and $V$, 
  there is a set $Y$ equal to one of $X \cap V$ or $V \del (X + X)$ for which $\mu_V(Y) < 2^{-d}$.
  If $Y = X \cap V$ then $\alpha(X) \ge \alpha_V(Y)$ and if $Y = V \del (X + X)$,
  then $\omega(X+X) \ge \omega((X+X) \cap V) = \alpha (V \del (X + X))) = \alpha_V(Y)$, 
  so $\max(\alpha(X), \omega(X+X)) \ge \alpha_V(Y)$ in both cases,
  and it suffices to show that $\alpha_V(Y) \ge Cn^{1/7}$. 
  
  If $Y = \varnothing$, then $\alpha_V(Y) = \dim(V) \ge Cn^{1/7}$ as required. 
  Otherwise, since $\mu_Y(V) < 2^{-d}$ we have $\cL(\mu_V(Y)) \ge d+1$,
  so Theorem~\ref{bbthm2} gives $\alpha_V(Y) + 1 \ge \cL(\mu_V(Y)) \ge d +1 \ge Cn^{1/7}+1$ and so
  $\alpha_V(Y) \ge Cn^{1/7}$.
\end{proof}

This easily implies the following estimate on the Ramsey number $R_2(2,t)$ 
as defined in Theorem~\ref{grthm}; this was also stated in [\ref{hp}].

\begin{theorem}\label{trifree}
  There is a constant $C$ so that $R_2(2,t) \le C t^7$ for all $t$. 
\end{theorem}

\begin{proof}[Proof of Theorem~\ref{trifree}]
  Let $C = C_{\ref{complorsum}}^{-1/7}$. 
  Let $s,n \in \bN$ with $n \ge Cs^7$. 
  If the theorem fails, there is a red-blue colouring of $G \cong \PG(n-1,2)$ 
  with no red two-dimensional subspace and no blue $s$-dimensional subspace. 
  Choose such a colouring with as many red points as possible, 
  and let $R$ and $B$ be the sets of red and blue points respectively, viewed as subsets of $\bF_2^n \del \{0\}$.
  Since there is no red triangle, we have $(R + R) \cap R = \varnothing$.
  By the maximality of $R$, each $e \in B$ must form a triangle 
  with two points in $R$, so $B \subseteq R + R$. 
  It follows that $(R + R) \del \{0\} = B$, so $\omega(R+R) = \omega(B) = \alpha(R)$.
  
  By this observation, Theorem~\ref{complorsum}, and the choice of $C$,
  we have $\omega(B) = \alpha(R) = \min(\alpha(R), \omega(R+R)) \ge C_{\ref{complorsum}}n^{1/7} \ge C_{\ref{complorsum}}(Cs^7)^{1/7} = s$, 
  so there is a blue $s$-dimensional subspace, giving a contradiction. 
\end{proof}

Analogously to [\ref{hp}, Conjecture 1], 
we conjecture that the right bound in Theorem~\ref{complorsum} is linear. 
\begin{conjecture}
  There is a constant $C > 0$ such that, for all $n \in \bN$,
   every $X \subseteq \bF_2^n$ satisfies $\max(\alpha(X), \omega(X+X)) \ge C n$.
\end{conjecture}

We can now prove Theorem~\ref{mainnonrainbow}.

\begin{theorem}\label{nonrainbow}
  If $c_0$ is a non-rainbow $s_0$-colouring of $T \cong \PG(1,2)$, and $s \ge s_0$, 
  then Conjecture~\ref{multicolouredeh} holds for $c_0$ and $s$, 
  with $C = C_{\ref{complorsum}}$ and $\delta = \tfrac{1}{7}$. 
\end{theorem}
\begin{proof}
  Let $C = C_{\ref{complorsum}}$. 
  Since $c_0$ is not rainbow, we can write $T = \{e,e',f\}$, where 
  $c_0(e) = c_0(e') = \kappa$ and $c_0(f) = \lambda$ for some colours $\kappa$ and $\lambda$. 
  Let $n \in \bN$, and $c$ be a colouring of $G \cong \PG(n-1,2)$.
  Let $X = \{x \in G \colon c(x) = \kappa\}$, seen as a subset of $\bF_2^n \del \{0\}$. 
  By Theorem~\ref{complorsum}, we either have $\alpha(X) \ge Cn^{1/7}$ or $\omega(X+X) \ge Cn^{1/7}$. 
  
  If $\alpha(X) \ge Cn^{1/7}$, then there is a subspace $W$ of $\bF_2^n$ containing 
  no point coloured $\kappa$, for which $\dim(W) \ge Cn^{1/7}$. 
  Thus, $W$ gives a $(c,s)$-homogeneous subspace of the same dimension in $G$. 

  If $\omega(X+X) \ge Cn^{1/7}$, then there is a subspace $W$ of $\bF_2^n$ with $\dim(W) \ge Cn^{1/7}$
  such that for all $u \in W \del \{0\}$, there exist $x,x' \in X$ with $x + x' = u$. 
  If $c(u) = \lambda$, then since $c(x) = c(x') = \kappa$, 
  the triple $\{x,x',u\}$ corresponds to a copy of $c_0$ in $c$. 
  Since $c$ is $c_0$-free, this does not occur, 
  so there is no element of $W \del \{0\}$ coloured $\lambda$ by $c$. 
  Therefore $W$ gives a $(c,s)$-homogeneous subspace in $G$, as required. 
\end{proof}

\section{Modularity}

  Our main results will be applied to projective geometries over finite fields, 
  but it is convenient to work at the more abstract level of modular matroids. 
  The consequence is that our proofs are co-ordinate free, 
  and apply to a slightly larger class of objects. 
  
  A matroid $M$ is \emph{modular} if every pair $F,F'$ of flats of $M$ satisfies 
  $r_M(F) + r_M(F') = r_M(F \cup F') + r_M(F \cap F')$, 
  or equivalently if $r_M(\cl_M(A) \cap \cl_M(B)) = \sqcap_M(A,B)$ for all $A,B \subseteq E(M)$.   
  
  Canonical examples include matroids of rank at most two, free matroids, 
  projective planes in rank three (including non-Desarguesian ones), 
  and projective geometries over finite fields.
  Indeed, the simple connected modular matroids of rank at least four are precisely those 
  of the form $\PG(n-1,q)$ for a prime power $q$ (See [\ref{oxley}, Proposition 6.9.1]). 
  
  Even though Theorem~\ref{mainrainbow} only applies to binary projective geometries, 
  Theorem~\ref{mainstructure} applies to all connected modular matroids,
  which by the aforementioned classification comprise precisely the loopless matroids 
  whose simplification is either a rank-$2$ uniform matroid, a projective plane, 
  or $\PG(n-1,q)$ for $n \ge 4$. 

  The first four lemmas in this section follow easily from the classification
  just discussed, but for completeness, we include self-contained proofs 
  that use the definition of modularity directly. 

  \begin{lemma}
    If $M$ is a loopless modular matroid, then flats $F,F'$ of $M$ are complementary if and only 
    if they are disjoint and their union spans $M$. 
  \end{lemma}
  \begin{proof}
    It suffices to show that two flats $F_1,F_2$ have local connectivity zero if and only
    if they are disjoint. But $r_M(F_1 \cap F_2) = \sqcap_M(F_1,F_2)$, and since $M$ is loopless, 
    we have $r_M(F_1 \cap F_2) = 0$ if and only if $|F_1 \cap F_2| = 0$. The equivalence holds. 
  \end{proof}

  \begin{lemma}\label{modularcontract}
    If $M$ is a modular matroid, then $M \con C$ is modular for each set $C \sse E(M)$, and $M | F$ is modular for each flat $F$ of $M$. 
  \end{lemma}
  \begin{proof}
    The fact that $M | F$ is modular follows immediately from the definition, since the flats of $M | F$ are the flats of $M$ 
    that are contained in $F$. 
    To see that $M \con C$ is modular for all $C$, let $X, Y \subseteq E(M \con C)$.
    By the modularity of $M$ and the fact that $C \subseteq \cl_M(X \cup C) \cap \cl_M(Y \cup C)$, 
    \begin{align*}
        \sqcap_{M \con C}(X,Y) &= \sqcap_M(X \cup C, Y \cup C) - r_M(C)\\ 
                               &= r_M(\cl_M(X \cup C) \cap \cl_M(Y \cup C)) - r_M(C) \\ 
                               &= r_{M \con C}(\cl_{M \con C}(X) \cap \cl_{M \con C}(Y)),
    \end{align*}
    so $M \con C$ is modular. 
  \end{proof}

  \begin{lemma}\label{modularconnectedcontract}
    If $F$ is a nonempty proper flat of a connected modular matroid $M$, then $M|F$ and $M \con F$ are both connected and modular. 
  \end{lemma}
  \begin{proof}
    We know both are modular from the previous lemma, so it suffices to show that both are connected. 
    To show that $M | F$ is connected, let $x, y \in F$. By the connectedness of $M$, 
    there is a circuit $C$ of $M$ containing $x$ and $y$.
    If $C = \{x,y\}$, then $C$ is a circuit of $M | F$ containing $x$ and $y$, as required. 
    Otherwise, both $\{x,y\}$ and $C - \{x,y\}$ are independent in $M$, so 
    \begin{align*}
      r_M(\cl_M(\{x,y\}) \cap \cl_M(C- \{x,y\})) &= \sqcap_M(C - \{x,y\}, \{x,y\}) \\ 
        &= (|C|-2) + 2 - (|C|-1) = 1.
    \end{align*}
    Therefore $M$ has a nonloop $e$ spanned by both $\{x,y\}$ and $C - \{x,y\}$. Since $F$ is a flat, $e \in F$.
    If $e$ is parallel to some $f \in \{x,y\}$, then the independent set $(C - \{x,y\}) \cup \{f\}$ 
    is spanned by its proper subset $C - \{x,y\}$, a contradiction. 
    Therefore $e$ is parallel to neither $x$ nor $y$ but is spanned by $\{x,y\}$, 
    so $\{e,x,y\}$ is a circuit contained in $F$.  Therefore $M | F$ is connected. 

    To see that $M \con F$ is connected, let $F_0$ be a complement of $F$, so $F_0$ is spanning in $M \con F$. 
    Since $F$ is a flat, the matroid $M \con F$ is loopless and has a connected spanning restriction $M|F_0$, so is connected. 
  \end{proof}

  The next lemma is a well-known result about finite projective planes, stated in matroidal language. 

  \begin{lemma}\label{rankthreemodular}
    If $M$ is a connected, modular rank-three matroid,
    then there is an integer $q \ge 2$ such that every line of $M$ has length $q+1$, 
    every point is contained in exactly $q+1$ lines, and the number of points and the 
    number of lines are both equal to $q^2 + q + 1$.
  \end{lemma}
  \begin{proof}
    We may assume that $M$ is simple. By modularity, any two distinct lines intersect in a single element. 
    Let $\cL$ be the set of lines of $M$, and $L, L' \in \cL$ be distinct lines intersecting in a point $e$. 
    Let $f \in L - e$ and $f' \in L' - e$, noting that $f \ne f'$. 
    
    If $E(M) = L \cup L'$, then since $M$ is coloopless, there are distinct $x,x' \in L - L'$
    and distinct $y,y' \in L' - L$. Then $\cl_M(\{x,y\})$ and $\cl_M(\{x',y'\})$ intersect
    in some element $z$, and it is easy to argue that $z \notin L \cup L'$, a contradiction. 
    Therefore there exists $g \in E(M) - (L \cup L')$. 

    Define $\varphi : L \to L'$ by, for each $x \in L$, setting $\varphi(x)$ 
    to be the unique point in $L' \cap \cl_M(\{x,g\})$. 
    Since $x$ is the only element of $L \cap \cl_M(\{\varphi(x), g'\})$, 
    the function $\varphi$ is injective. 
    It follows that $|L'| \ge |L|$. The choices of $L$ and $L'$ were arbitrary, 
    so all lines of $M$ have the same length $q + 1$ for some $q \ge 1$. 
  
    Let $e \in E(M)$. Since $r(M) = 3$, there is a line $L$ not containing $e$. 
    The collection $\cL_e = \{\cl_M(\{e,x\}) : x \in L\}$ contains $|L| = q+1$ distinct lines, 
    and since every line through $e$ intersects $L$, each line through $e$ 
    belongs to $\cL_e$. Therefore $e$ is in exactly $|\cL_e| = q+1$ lines. 
  
    Moreover, the lines in $\cL_e$ partition $E(M) - \{e\}$ and each has length $q+1$, 
    so $|M| = 1 + \sum_{L \in \cL_e} |L - \{e\}| = 1 + (q+1)(q+1-1) = q^2+q+1$. 
  
    By double-counting the set of all pairs $(e, L)$ for which $e \in L \in \cL$, 
    we see that $|M|(q+1) = (q+1)|\cL|$, since each point is in $q+1$ lines and each line has $q+1$ points. 
    Therefore $|\cL| = |M| = q^2 + q+ 1$. 

    It remains to argue that $q \ge 2$. Indeed, if $q = 1$, then $|M| = 3 = r(M)$, 
    so $M$ has a coloop, which contradicts that $M$ is connected. 
  \end{proof}

  The next lemma shows that modular matroids enjoy a property sometimes
  known as \emph{roundness}. 

  \begin{lemma}\label{modularround}
    If $M$ is a connected modular matroid, then every cocircuit of $M$ is spanning. 
  \end{lemma}
  \begin{proof}
    Let $K$ be a cocircuit of $M$, and let $x \in K$.
    For each $e \notin K$, the rank-two restriction $M | \cl_M(\{e,x\})$ is connected by Lemma~\ref{modularconnectedcontract}, 
    so there is some $f$ for which $T = \{e,f,x\}$ is a triangle of $M$. 
    Since $T$ is a circuit, we have $|T \cap K| \ne 1$, so $f \in K$. 
    It follows that $x \in \cl_M(\{e,f\}) \subseteq \cl_M(K)$. This holds for all $x \in E(M)-K$, 
    so $E(M)-K \subseteq \cl_M(K)$, and the lemma follows. 
  \end{proof}

   Our last lemma in the section exploits the abundance of triangles in modular matroids
   to find such a triangle that hits three particular sets. 

  \begin{lemma}\label{basispartition}
    Let $H$ be a hyperplane of a connected modular matroid $M$, 
    and let $B$ be a basis of $H$. If $(X,Y)$ is a partition of $E(M) - H$
    and the sets $B,X$ and $Y$ are nonempty, then $M$ has a triangle intersecting $X,Y$ and $B$.
  \end{lemma}  
  \begin{proof}
    Since $0 < |B| = r(M)-1$, we have $r(M) \ge 2$.
    If $r(M) = 2$, then $|B| = 1$, and $M \del \cl_M(B)$ has at least two points by connectedness. 
    Since $X$ and $Y$ are nonempty, it follows that there exists $x \in X$ and $y \in Y$
    not parallel, and $x,y,b$ is the required triangle for $b \in B$. 

    Suppose, therefore, that $r(M) \ge 3$, and inductively that the lemma holds for 
    matroids of rank $r(M)-1$. Let $A = E(M) - H$, let $e \in A$, and let $f \in B$. 
    Let $H_0 = \cl_M(B - f)$ and let $A_0 = \cl_M(H_0 \cup \{e\}) - H_0$, 
    noting that $A_0 \subseteq A$ and $H_0 \cup A_0$ is a hyperplane of $M$.
    If $X \cap A_0$ and $Y \cap A_0$ are both nonempty, 
    then the inductive hypothesis applied to the basis $B-f$ of $H_0$ in $M | (H_0 \cup A_0)$
    with the partition $(X \cap A_0, Y \cap A_0)$ gives the required triangle.
    We may therefore assume, say, that $X \cap A_0 = \es$, so $A_0 \subseteq Y$. 

    Let $x \in X$, so $x \in A - A_0$. By modularity, 
    the line $\cl_M(\{x,f\})$ intersects the hyperplane $A_0 \cup H_0$
    in a nonloop $y$; since $f \in H - H_0$ and $x \in A - A_0$, we know that $y \notin \{x,f\}$, 
    so $\{x,y,f\}$ is a triangle. 
    Since $\{x,f\} \cap A = \{x\}$, 
    but  the circuit $\{x,y,f\}$ does not intersect the cocircuit $A$ in exactly one element, 
    we have $y \in A$, so $y \in A_0 \subseteq Y$. 
    This implies that $\{x,y,f\}$ is the required triangle. 
  \end{proof}

\section{Coloured Matroids and Decomposers}

  A \emph{coloured matroid} is a pair $(M; c)$, where $M$ is a matroid and $c$ is a colouring of $E(M)$. 
  When discussing matroid properties such as connectedness, modularity, or independence of a set,
  we identify $(M;c)$ with the matroid $M$. 
  The colouring $c$ is \emph{simple} (with respect to $M$) if $M$ has no loop, and $c(e) = c(f)$
  for all $e, f$ that are parallel in $M$. Note that $M$ itself need not be simple for this to hold, 
  but simplicity of $M$ implies simplicity of $c$.
  A set $X$ is \emph{$c$-monochromatic} if $c$ is constant on $X$,
  is \emph{$c$-polychromatic} otherwise, and is
  \emph{$c$-rainbow} if $c$ is injective on $X$. 
  We omit the reference to $c$ if it is clear from context. 

  The following lemma shows that the existence of a rainbow triangle
  is equivalent to the existence of a larger rainbow circuit. 

  \begin{lemma}\label{rainbowcircuit}
      A  modular coloured matroid $(M;c)$ has a rainbow triangle if and only if
      $(M;c)$ has a rainbow circuit of size at least $3$.
  \end{lemma}
  \begin{proof}
    One direction is immediate. For the other direction, let $C$ be a minimal 
    rainbow circuit of $(M;c)$ for which $|C| \ge 3$. Towards a contradiction, suppose that $|C| \ge 4$.
    
    Let $(X_1,X_2)$ be a partition of $C$ with $|X_1|,|X_2| \ge 2$. 
    Since each $X_i$ is a proper subset of $C$, it is independent in $M$,
    which implies that 
    $r_M(\cl_M(X_1) \cap \cl_M(X_2)) = \sqcap_M(X_1,X_2) = |X_1| + |X_2| - r_M(C) = 1$, 
    
    Let $e$ be a nonloop in $\cl_M(X_1) \cap \cl_M(X_2)$. 
    For each $i$, the set $X_i \cup \{e\}$ is dependent, so contains a circuit $C_i$ with $e \in C_i$. Note that $|C_i| \geq 3$ for each $i$; otherwise, applying circuit elimination to $C_1$ and $C_2$ yields a circuit properly contained in $C$ (since $|X_1|, |X_2| \ge 2$).  
    Now, since $C_i - \{e\} \subseteq X_i$ and $X_1 \cup X_2$ is rainbow, 
    one of the circuits $C_1,C_2$ must be rainbow. But both are clearly smaller than $C$, 
    which contradicts minimality.
  \end{proof}

 We now present several equivalent characterizations of modular coloured matroids without rainbow triangles.  

  \begin{theorem}\label{easyequiv}
      Let $(M;c)$ be a modular coloured matroid for which $c$ is simple. The following are equivalent:
      
      \begin{enumerate}[(i)]
      \item\label{nrt} $(M;c)$ has no rainbow triangle;
      \item\label{nrc} $(M;c)$ has no rainbow circuit;
      \item\label{setsmall} $|c(X)| \le r_M(X)$ for all $X \subseteq E(M)$;
      \item\label{flatsmall} $|c(F)| \le r_M(F)$ for every flat $F$ of $M$.
      \end{enumerate}
  \end{theorem}
  \begin{proof}
    By the simplicity of $c$, every rainbow circuit of $M$ has size at least $3$, 
    so the equivalence of (\ref{nrt}) and (\ref{nrc}) follows from Lemma~\ref{rainbowcircuit}. 
    It is clear that  (\ref{setsmall}) implies (\ref{flatsmall}). 
    Moreover, if (\ref{flatsmall}) holds and $X \subseteq E(M)$, then $|c(X)| \le |c(\cl_M(X))| \le r_M(\cl_M(X)) = r_M(X)$, 
    so (\ref{flatsmall}) implies (\ref{setsmall}). 
    Since a rainbow triangle $T$ satisfies $|c(T)| > r_M(T)$, we also know that (\ref{setsmall}) implies (\ref{nrt}). 
    It is thus sufficient to show that (\ref{nrc}) implies (\ref{setsmall}). 
    
    Suppose that (\ref{nrc}) holds and (\ref{setsmall}) does not, and let $X$ be a minimal set for which $|c(X)| > r_M(X)$. 
    For each $e \in X$, the minimality of $X$ implies that 
    \[|c(X-e)| \le r_M(X-e) \le r_M(X) \le |c(X)| - 1 \le |c(X-e)|,\]
    so equality holds throughout, meaning that $r_M(X) = |c(X)|-1$ 
    and that $e$ is the only element of $X$ with colour $c(e)$. 
    This holds for all $e \in X$, so $X$ is rainbow, and $r_M(X) = |c(X)| - 1 = |X|-1$. 
    Therefore $X$ contains a rainbow circuit $C$, contradicting (\ref{nrc}).     
  \end{proof}

    A rainbow-triangle-free coloured matroid is essentially the same
    as one where every line contains at most two colours. 
    A natural generalization is the class of matroids where all
    rank-$k$ flats contain at most $k$ colours for some fixed $k$.
    However, we can see by way of the following easy result 
    that this does not make the classification problem any more interesting.

  \begin{theorem}
    Let $(M;c)$ be a coloured matroid, 
    and $k \ge 2$ be an integer. 
    If $r(M) \ge k$ and $|c(F)| \le k$ for every rank-$k$
    flat $F$ of $M$, then either $|c| \le k$, or $(M;c)$ has no rainbow triangle. 
  \end{theorem}
  \begin{proof}
  Suppose for a contradiction that $(M;c)$ has a rainbow triangle $T$, 
  but also that $|c| > k$. Then there is a $(k-2)$-element set 
  $S \subseteq E(M) - T$ such that $S \cup T$ is rainbow.
  Since \[r_M(S \cup T) \leq r_M(T)+|S|=2+(k-2)=k,\]
  $S \cup T$ is contained in a rank-$k$ flat of $M$ using at least $|S \cup T|=k+1$ colours,
  a contradiction. 
  \end{proof}

  Our main structural result, Theorem~\ref{mainstructuresimple}, 
  adds a nontrivial fifth condition to the list in Theorem~\ref{easyequiv}. 
  To be able to state and prove this condition, we need a notion of decomposition for coloured matroids. 
  A \emph{decomposer} of a coloured matroid $(M; c)$ is a proper flat $F$ of $M$ such that, for all $e \in E(M) - F$,
  the set $\cl_M(F \cup \{e\}) - F$ is $c$-monochromatic. 
  
  It is easy to see that, if $M$ has no loop, the flat $\es$ is a decomposer if and only if 
  $c$ is simple. A decomposer of rank at least one is \emph{nontrivial}.
  For nearly all of our theorems, the matroids addressed are connected
  and have rank at least one.
  This means they are loopless, that points are the same as parallel classes, 
  and that nontrivial decomposers are precisely nonempty decomposers. 
  
  However, our inductive strategy involves contractions that create parallel pairs;  
  to facilitate this, we work at a level of generality that includes non-simple matroids. 
  The reader should bear in mind the possibility that 
  two distinct elements are parallel (and therefore not in a rainbow triangle), 
  which can affect the course of our arguments. 

  For a colouring $c$ of a set $X$, we write $c | Y$ for the restriction of $c$ to domain $Y \subseteq X$. 
  If $X \subseteq E(M)$, then we write $(M; c) \con X$ for the coloured matroid $((M \con X); c | E(M \con X))$, 
  and $M | X$ for the coloured matroid $(M|X; c |X)$. 
  
  \begin{lemma}\label{decompdecomp}
    Let $(M; c)$ be a coloured matroid. If $F$ is a decomposer of $(M; c)$, and $F_0$ is a decomposer of $(M; c) | F$, 
    then $F_0$ is also a decomposer of $(M; c)$. 
  \end{lemma}  
  \begin{proof}
    Clearly $F_0$ is a proper flat of $M$. 
    Let $e \in E(M) - F_0$. If $e \in F$, then $\cl_M(F_0 \cup \{e\}) - F_0 = \cl_{M | F}(F_0 \cup \{e\}) - F_0$, which is monochromatic. 
    If $e \notin F$, then $\cl_M(F_0 \cup \{e\}) \cap F = F_0$, and so 
    $\cl_M(F_0 \cup \{e\}) - F_0 \subseteq \cl_M(F \cup \{e\}) - F$, which is monochromatic, as required. 
  \end{proof}
  
  \begin{lemma}\label{decompcontract}
    Let $F$ be a flat of a coloured matroid $(M; c)$. 
    If $F_0$ is a decomposer of $(M; c) \con F$, 
    then $\cl_M(F \cup F_0)$ is a decomposer of $(M; c)$. 
  \end{lemma}
  \begin{proof}
    Since $F_0$ is a proper flat of $M \con F$, 
    we see that $\cl_M(F \cup F_0)$ is a proper flat of $M$. 
    Now let $e \in E(M) - \cl_M(F \cup F_0)$. Then 
    \[\cl_M(F \cup F_0 \cup \{e\}) - \cl_M(F \cup F_0) \sse \cl_{M \con F}(F_0 \cup \{e\}) - F_0,\] 
    which is monochromatic. 
  \end{proof}

We now proceed to show that every connected, modular rainbow-triangle-free
coloured matroid of rank at least two has a decomposer. 
We first handle the rank-two case, which is an easy case analysis. 

\begin{lemma}\label{ranktwo}
  Every connected coloured matroid $(M ;c)$ with $r(M) = 2$ and $|c| \ge 3$
  has a rainbow triangle or a nonempty decomposer.
\end{lemma}
    
\begin{proof}
  Suppose that neither holds. 
  Let $\cP$ be the set of points of $M$, and for each $i$, let $\cP_i = \{P \in \cP : |c(P)| = i\}.$
  If there exist $P,Q \in \cP_1$ with $c(P) \ne c(Q)$, then since $|c| \ge 3$
  there is some $x$ with $c(x) \notin c(P) \cup c(Q)$, 
  and $\{e,f,x\}$ is a rainbow triangle for any $e \in P$ and $f \in Q$, a contradiction. 
  So any two points in $\cP_1$ have the same colour. 

  It follows that, if $\cP - \{P\} \subseteq \cP_1$ for some $P \in \cP$, then $P$ is a decomposer of $M$, 
  a contradiction. So there are distinct points $P, Q \in \cP \del \cP_1$. 
  Note that for all $(\kappa,\lambda) \in c(P) \times c(Q)$ with $\kappa \ne \lambda$,
  we have $c(E - P \cup Q) \subseteq \{\kappa, \lambda\}$, as otherwise there is a rainbow triangle. 
  Since $P, Q \notin \cP_q$, such a pair $(\kappa, \lambda)$ exists,
  so $[c] \subseteq c(P \cup Q)$ and therefore $|c(P \cup Q)| \ge 3$.
  It follows in turn that $c(P) \times c(Q)$ contains distinct pairs together comprising at least 
  three distinct colours $\kappa, \lambda, \mu$, which implies that 
  $c(E - P \cup Q) \subseteq \{\kappa, \lambda\} \cap \{\kappa, \mu\} \cap \{\lambda, \mu\} = \es$, 
  so $E - P \cup Q = \es$, which contradicts the connectedness of $M$. 
\end{proof}

We now prove a lemma to facilitate induction in the rank-three case.

\begin{lemma}\label{rankthree}
  Let $(M; c)$ be a connected, modular coloured matroid 
  with no rainbow triangle. 
  If $r(M) = 3$ and $|c| \ge 3$, then $M$ has a line $L$ 
  for which $(M; c) | L$ has a nonempty decomposer. 
\end{lemma}
\begin{proof}    
  Suppose not, and let $M$ be a counterexample for which $|M|$ is minimized.
  Let $\cL$ be the set of lines of $M$. 
  By Lemma~\ref{ranktwo}, we have $|c(L)| \le 2$ for all $L \in \cL$. 

  \begin{claim}
    The colouring $c$ is simple. 
  \end{claim}
  \begin{subproof}
    If some point $P$ is polychromatic, then since $|c| \ge 3$, 
    there is some $x$ for which $|c (P \cup \{x\})| \ge 3$,
    and now any line $L$ containing $P$ and $x$ satisfies $|c(L)| \ge 3$, 
    a contradiction. 
  \end{subproof}

  If $\{e,f\}$ is a parallel pair of $M$, then since $c(e) = c(f)$,
  the matroid $(M; c) \del e$ is also a counterexample, a contradiction to minimality. 
  Therefore $M$ is simple. 
  Thus, Lemma~\ref{rankthreemodular} gives that there is some $q \ge 2$ 
  such that every line of $M$ contains $q+1$ elements, 
  each element of $M$ lies in $q+1$ lines, and $|E(M)| = |\cL| = q^2+q+1$.

  \begin{claim}\label{twocollines}
    For each $L \in \cL$, 
  there are distinct colours $i,j$ such that $c(L) = \{i,j\}$,
  and $L$ contains at least two elements of each colour. 
  \end{claim}
  \begin{subproof}
    We know that $|c(L)| \le 2$, but since $(M; c) | L$ has no rank-$1$ decomposer, 
    we also have $|c(L -e)| \ge 2$ for each $e \in L$. The claim follows. 
  \end{subproof}

  \begin{claim}
    $|c| = 3$. 
  \end{claim}
  \begin{subproof}
    If $|c| \ge 4$, we can find $L,L' \in \cL$ with $|c(L \cup L')| \ge 4$. 
    But $|c(L \cup L')| = |c(L)| + |c(L')| - |c(L \cap L')| \le 2 + 2 -1 = 3$. 
  \end{subproof}

  Assume for simplicity that $c(M) = [3]$.
  For each $i \in [3]$, let $X_i = c^{-1}(\{i\})$, 
  so $|X_1| + |X_2| + |X_3| = |M| = q^2 + q + 1$.
  
  For all distinct $i,j \in [3]$, write $\cL_{ij}$ for the set of lines $L \in \cL$ with $c(L) = \{i,j\}$. 
  By~\ref{twocollines}, we have $|\cL_{12}| + |\cL_{23}| + |\cL_{13}| = |\cL| = q^2 + q + 1$. 
  If $L_{ij}$ is empty for some $i \ne j$, then either $X_i = \es$ or $X_j = \es$, contradicting $|c| = 3$. 
  Therefore $|\cL_{ij}| > 0$ for all $i \ne j$. 
  
  \begin{claim}\label{ballinediff}
    Let $i,j,j' \in [3]$ be distinct. If $L \in \cL_{ij}$ and $L' \in \cL_{ij'}$, 
    then $|L \cap X_i| = |L' \cap X_i|$. 
  \end{claim}
  \begin{subproof}
    By symmetry, it suffices to show that $|L \cap X_i| \le |L' \cap X_i|$. 
    Let $\{e\} = L \cap L'$ (so $c(e) = i$), and let $f \in L$ with $c(f) = j$ 
    and $f' \in L'$ with $c(f') = j'$. Then $\cl_M(\{f,f'\}) \in \cL_{jj'}$.
    By~\ref{twocollines}, there exists $g \in \cl_M(\{f,f'\}) - \{f,f'\}$ 
    with $c(g) = j$. Note that $g \notin L \cup L'$.  

    Define a function $\varphi : L \to L'$ by setting $\varphi(x)$ to be the 
    unique element in $L' \cap \cl_M(\{x,g\})$. (The set $\cl_M(\{x,g\})$ is 
    always a line distinct from $L'$, since $g \notin L \cup L'$.)
    If $\varphi(x) = \varphi(x') = y$, then the line $\cl_M(\{g,y\})$ intersects 
    $L$ in both $x$ and $x'$, which implies that $x = x'$. Thus $\varphi$ is injective. 
    
    Note that $\varphi(e) = e$. If $x \in L-e$ with $c(x) = i$, 
    then since $\{x, g, \varphi(x)\}$ is a triangle and $c(g) = j$, 
    we have $c(\varphi(x)) \in \{i,j\}$, and since 
    $c(\varphi(x)) \in L' \in \cL_{ij'}$, it follows that $c(\varphi(x)) = i$. 

    Thus, $\varphi(L \cap X_1) \subseteq (L' \cap X_1)$, and  
    the claim follows by injectivity.   
  \end{subproof}

  \begin{claim}\label{balllinesame}
    If $i \ne j$ and $L, L' \in \cL_{ij}$, then $|L \cap X_i| = |L' \cap X_i|$. 
  \end{claim}
  \begin{proof}
    Let $k \in [3] - \{i,j\}$, and let $L_0 \in \cL_{ik}$. By~\ref{ballinediff}, 
    we have $|L \cap X_i| = |L_0 \cap X_i| = |L' \cap X_i|$, as required. 
  \end{proof}

  By the last two claims, there exist\textbf{} integers $n_1, n_2, n_3$ so that, for all $i \ne j$ and $L \in \cL_{ij}$, 
  we have $|L \cap X_i| = n_i$ and $|L \cap X_j| = n_j$. It follows that $n_1 + n_2 = n_1 + n_3 = n_2 + n_3 = q+1$, 
  whence $n_1 = n_2 = n_3 = \tfrac{1}{2}(q+1)$. So every line in $\cL_{ij}$ contains 
  exactly $\tfrac{1}{2}(q+1)$ points of each colour.

  Let $t \in [3]$ be such that $|X_t|$ is minimized, so $|X_t| \le \tfrac{1}{3}(q^2 + q + 1)$.
  Let $x \in X_t$. Each of the $q+1$ lines through $x$ is in $\cL_{ti}$ for some $i \neq t$, 
  so contains $\tfrac{q+1}{2} - 1$ points in $X_t$ other than $x$. These lines intersect only at $x$, 
  so 
  \[ \tfrac{1}{3}(q^2 + q+ 1) \ge |X_t| \ge 1 + (q+1) \left(\tfrac{q+1}{2}-1\right) = \tfrac{1}{2}(q^2+1),\] 
  which yields the contradiction $q \le 1$.
\end{proof}

We now show that a decomposer for a hyperplane can be contracted while preserving
the property of having no rainbow triangle. 

\begin{lemma}\label{contracthpdecomposer}
  Let $(M; c)$ be a connected, modular coloured matroid
  with no rainbow triangle. 
  Let $H$ be a hyperplane of $M$. If $F$ is a decomposer of $(M; c) | H$,
  then $(M; c) \con F$ has no rainbow triangle. 
\end{lemma}
\begin{proof}
  By Lemma~\ref{modularconnectedcontract}, the matroid $M \con F$ is connected and modular. 
  Suppose that $(M; c) \con F$ has a rainbow triangle $T = \{x,y ,z\}$, 
  chosen so that $z \in H$ if possible.  
  Since $H - F$ is a hyperplane of the modular matroid $M \con F$, 
  there exists $w$ in $\cl_{M \con F}(T) \cap H$, 
  and since $T$ is rainbow, the set $T \cup \{w\}$ contains a rainbow 
  triangle of $(M; c) \con F$ containing $w$.
  Thus there exists a rainbow triangle of $M \con F$ intersecting $H$;  
  by the choice of $T$, it follows that $z \in H$. Now
  \begin{align*} 
      \sqcap_M(\{x,y\}, F \cup \{z\}) &= r_M(\{x,y\}) - r_{M \con (F \cup \{z\})}(\{x,y\}) \\
        &=  2 - (r_{M \con F} (T) - 1) = 1,
  \end{align*}
  so by modularity, there exists $z' \in \cl_M(F \cup \{z\}) \cap \cl_M(\{x,y\})$. 
  If we had $z' \in F$, then since $z' \in \cl_M(\{x,y\})$, we would have
  \[2 = r_M (\{x,y\}) > r_{M \con z'} (\{x,y\}) \ge r_{M \con F} (\{x,y\}) = 2,\]
  a contradiction. Therefore $z' \in \cl_M(F \cup \{z\}) - F$, and so $z$ and $z'$ are 
  parallel in $M \con F$, implying that $z'$ is parallel to neither $x$ nor $y$ in $M$. 
  Since $z' \in \cl_M(\{x,y\})$, It follows that $\{x,y,z'\}$ is a triangle of $M$. 
  
  Because $F \cup \{z\} \subseteq H$ and $F$ is a decomposer of $(M; c) | H$, 
  the set $\cl_M(F \cup \{z\}) - F$, which contains $z$ and $z'$, is monochromatic. 
  Therefore $c(z') = c(z)$. Since $\{x,y,z\}$ is rainbow, 
  so is the triangle $\{x,y,z'\}$ of $M$, giving a contradiction.  
\end{proof}

We now prove the decomposition theorem that will imply our main
structure theorem for rainbow-triangle-free matroids. 

\begin{theorem}\label{main1}
  Let $(M; c)$ be a connected, modular coloured matroid with no rainbow triangle. 
  If $r(M) \ge 2$ and $|c| \ge 3$, then $(M; c)$ has a nonempty decomposer. 
\end{theorem}
\begin{proof}
  Suppose not, and let $(M ; c)$ be a counterexample chosen so that $r(M)$ is minimized.

  \begin{claim}
    There is a hyperplane $H$ of $M$ for which $(M; c) | H$ has a nonempty decomposer.
  \end{claim}
  \begin{subproof}
    If $M$ has a hyperplane $H$ for which $|c(H)| \ge 3$, then by the minimality in the choice of $M$,
    the restriction $(M; c) | H$ has a proper nonempty decomposer, as required. 

    Assume, therefore, that every hyperplane $H$ satisfies $|c(H)| \le 2$. 
    Let $x_1,x_2,x_3$ be elements of different colours.
    Since no hyperplane contains all three, we have $r(M) \le r(\{x_1,x_2,x_3\}) \le 3$. 
    If $r(M) = 2$, then Lemma~\ref{ranktwo} contradicts the choice of $M$ 
    as a counterexample. 
    So $r(M) = 3$, and the claim follows from Lemma~\ref{rankthree}.
  \end{subproof}

  Let $H$ be a hyperplane of $M$ for which $(M; c) |H$ has a nonempty decomposer, 
  and let $F$ be a minimal nonempty decomposer of $(M; c) |H$. 

  \begin{claim}\label{nodecompF}
    $(M; c) | F$ has no nonempty decomposer. 
  \end{claim}
  \begin{subproof}
    If $F_0$ were such a decomposer, then since $F$ is a decomposer of $(M; c) |H$, 
    by Lemma~\ref{decompdecomp} we see that $F_0$ is a decomposer of $(M; c) | H$
    that is properly contained in $F$, a contradiction to the minimality in the choice of $F$.
  \end{subproof}

  Since $r(M | F) < r(M)$, and the modular, connected coloured matroid $(M; c) | F$ has no rainbow triangle but is not a counterexample, 
  the minimality in the choice of $M$ immediately gives 
  \begin{claim}\label{mintwocolour}
    $r_M(F) = 1$ or $|c(F)| \le 2$.
  \end{claim}

  Since $F$ is not a decomposer of $M$, there exists $e \in E(M) - F$ 
  so that the set $A = \cl_M(F \cup \{e\}) - F$ contains elements of two different colours. 
  Note that $F \cup A$ is a flat and that $F$ is a hyperplane of $M | F \cup A$. 

  \begin{claim}\label{affinecolsubset}
    $c(F) \sse c(A)$. 
  \end{claim}
  \begin{subproof}
    Let $\beta \in c(F)$. If $c^{-1}(\beta)$ is not spanning in $F$, 
    then $r_M(F) \ge 2$ and $F$ has a hyperplane $F_0$ such that 
    $c(F - F_0) \subseteq c(F) - \{\beta\}$, 
    which by~\ref{mintwocolour} implies that $F-F_0$ is monochromatic. 
    This implies that $F_0$ is a nonempty decomposer of $(M; c) | F$, a contradiction. Therefore $c^{-1}(\beta)$ is spanning in $F$. 
    Let $B$ be a basis for $F$ with $c(B) = \{\beta\}$. 

    By its choice, the set $A$ is not monochromatic. Let $(X,Y)$ be a partition of $A$
    with $X,Y$ nonempty so that $c(X)$ and $c(Y)$ are disjoint. By Lemma~\ref{basispartition},
    there is a triangle of $M$ intersecting $X,Y$ and $B$. This triangle 
    is not rainbow, which implies that $\beta \in c(X) \cup c(Y) \subseteq c(A)$. 
    This holds for all $\beta \in c(F)$, so the claim follows. 
  \end{subproof}

  \begin{claim}
    $(M; c) \con F$ has a nonempty decomposer. 
  \end{claim}
  \begin{subproof}
    By Lemma~\ref{contracthpdecomposer}, we see that $(M;c) \con F$ has no rainbow triangle. 
    Since $F$ is a proper subflat of the hyperplane $H$, we have $r(M \con F) \ge 2$, 
    so by induction it suffices to show that $c(M \con F) = c(M)$. 
    Indeed, since $A \subseteq E(M) - F$, 
    we have $c(M \con F) \supseteq c(A) \supseteq c(F)$ by~\ref{affinecolsubset}, 
    so $c(M \con F) = c(M \con F) \cup c(F) = c(M)$.
  \end{subproof}

  Now Lemma~\ref{decompcontract} implies that $M$ has a nonempty decomposer, 
  contrary to its choice as a counterexample. 
\end{proof} 

\section{Lift-Joins}

To turn Theorem~\ref{main1} into a structure theorem, 
we need an operation that can be used to glue two rainbow-triangle-free colourings 
together into a larger one.

We now describe this operation, 
then use Theorem~\ref{main1} to derive Theorem~\ref{mainstructure}.
Theorem~\ref{main1} is the hard part of the actual proof; 
the work in this section amounts to straightforward manipulation of 
definitions in order that 
Theorem~\ref{mainstructure} can be carefully stated and proved. 

We call the gluing operation a \emph{lift-join}.
This was first defined in a slightly different (two-coloured) context in [\ref{bkknp}]. 
The definition uses the fact that, if $F_1$ and $F_2$ are skew flats of 
a modular matroid $M$ and $e \in \cl_M(F_1 \cup F_2) - F_1$, then 
the flats $\cl_M(F_1 \cup \{e\})$ and $F_2$ intersect in a point. 
A lift-join combines colourings of $F_1$ and $F_2$ to obtain a colouring of the closure of their union. 

\begin{definition}
  Let $F_1$ and $F_2$ be skew flats of a modular matroid $M$, 
  and let $c_1, c_2$ be colourings of $F_1, F_2$ respectively with $c_2$ simple in $M$. 
  
  The \emph{lift-join} $c_1 \otimes c_2$ 
  is the colouring $c$ of $\cl_M(F_1 \cup F_2)$
  such that $c(e) = c_1(e)$ for all $e \in F_1$, 
  and for each $e \in \cl_M(F_1 \cup F_2) - F_1$, 
  we have $\{c(e)\} = c_2(\cl_M(F_1 \cup \{e\}) \cap F_2)$. 
\end{definition}

We establish some basic properties of lift-joins. 

\begin{lemma}\label{liftjoinbasic}
  Let $F_1$ and $F_2$ be skew flats of a modular matroid $M$,
  and let $c_1, c_2$ be colourings of $F_1, F_2$ respectively with $c_2$ simple. Then
  \begin{enumerate}[(i)]
    \item $[c_1 \otimes c_2] = [c_1] \cup [c_2]$,
    \item\label{ljsimple} If $c_1$ is simple in $M$, then so is $c_1 \otimes c_2$,
    \item\label{ljrestrict} $(c_1 \otimes c_2) | F_i = c_i$ for each $i \in [2]$,
    \item\label{ljcomp} $d \circ (c_1 \ls c_2) = (d \circ c_1) \ls (d \circ c_2)$ for every 
      function $d$ with domain containing $[c_1] \cup [c_2]$,
    \item\label{ljdecomposer} If $r_M(F_2) > 0$, then $F_1$ is a decomposer of the coloured matroid 
      $(M | \cl_M(F_1 \cup F_2); c_1 \otimes c_2)$. 
    
  \end{enumerate}
\end{lemma}
\begin{proof}
  The first four facts follow immediately from the definition. 

  To see (\ref{ljdecomposer}), suppose that $r(F_2) > 0$, so $F_1$ is a proper flat of $M | \cl_M(F_1 \cup F_2)$. 
  Let $x \in \cl_M(F_1 \cup F_2) - F_1$, and $y$ be a nonloop in $\cl_M(F_1 \cup \{x\}) \cap F_2$. 
  By construction, we have $c(x') = c(y)$ for all $x' \in \cl_M(F_1 \cup \{x\}) - F_1$, 
  so $\cl_M(F_1 \cup \{x\}) - F_1$ is monochromatic, as required. 
\end{proof}

In fact, appropriately phrased, the implication in Lemma (\ref{liftjoinbasic})(\ref{ljdecomposer}) holds in both 
directions. 
\begin{lemma}\label{liftjoiniff}
  Let $(M; c)$ be a modular coloured matroid, 
  and $F_1$ and $F_2$ be complementary flats of $M$.
  The following are equivalent: 
  \begin{itemize}
    \item $F_1$ is a decomposer of $(M; c)$,
    \item $F_1$ is a proper flat, $c | F_2$ is simple, and $c = (c | F_1) \otimes (c | F_2)$. 
  \end{itemize}
\end{lemma}
\begin{proof}
    Suppose first that $F_1$ is a decomposer of $(M;c)$. By definition, $F_1$ is a proper flat. 
    If $e,e' \in F_2$ are parallel, then since $e' \in \cl_M(F_1 \cup \{e\}) - F_1$, we have $c(e) = c(e')$, 
    so $c | F_2$ is simple. For each $e \in F_1$ we have $c(e) = (c | F_1) (e)$. 
    For each $e \in \cl_M(F_1 \cup F_2) - F_1$, let $P = \cl_M(F_1 \cup \{e\}) \cap F_2$; this is a point of $M$. 
    Since $P \subseteq \cl_M(F_1 \cup \{e\})$, we know that $(c | F_2)(P) = c(P) = \{c(e)\}$, 
    and therefore $c = (c | F_1) \otimes (c | F_2)$.

    Conversely, suppose that $F_1$ is a proper flat, $c | F_2$ is simple, and $c = (c | F_1) \ls (c | F_2)$. 
    Since $r_M(F_2) = r(M) - r_M(F_1) > 0$ and $\cl_M(F_1 \cup F_2) = E(M)$, Lemma~\ref{liftjoinbasic}(\ref{ljdecomposer})
    implies that $F_1$ is a decomposer of $M$. 
\end{proof}

Recall that sets $X_1, \dotsc, X_t$ are \emph{mutually skew} in $M$ if $r_M(\cup_i X_i) = \sum_i r_M(X_i)$, 
or equivalently if $X_i$ is skew in $M$ to $\cup_{j < i} X_j$ for all $i \in \{1, \dotsc, t\}$. 
If $c_1, \dotsc, c_t$ are colourings of mutually skew flats $F_1, \dotsc, F_t$ of $M$
such that $c_i$ is simple for all $i > 1$, 
then we can iterate lift-joins to obtain a colouring $\bigotimes_i c_i$ of $\cl_M(\cup_i F_i)$. 
More precisely, we have $\bigotimes_i c_i = (\dotsc (((c_1 \otimes c_2) \otimes c_3) \otimes \dotsc )\otimes c_t)$. 
However, the next lemma proves associativity, telling us that the parentheses can be omitted. 
Note that the right-hand-side is well-defined by Lemma~\ref{liftjoinbasic}(\ref{ljsimple}).

\begin{lemma}\label{assoc}
  Let $F_1,F_2,F_3$ be mutually skew flats of a modular matroid $M$, and $c_1, c_2, c_3$ be colourings
  of $F_1,F_2,F_3$ respectively with $c_2$ and $c_3$ simple. 
   Then $(c_1 \otimes c_2) \otimes c_3 = c_1 \otimes (c_2 \otimes c_3)$. 
\end{lemma}
\begin{proof}
  Let $e \in \cl_M(F_1 \cup F_2 \cup F_3)$. If $e \in F_1$, then $e \in \cl_M(F_1 \cup F_2)$, so 
    \[((c_1 \otimes c_2) \otimes c_3)(e) = (c_1 \otimes c_2)(e) = c_1(e) = (c_1 \otimes (c_2 \otimes c_3))(e). \] 
  If $e \in \cl_M(F_1 \cup F_2) - F_1$, then let $f \in F_2$ be parallel to $e$ in $M \con F_1$. 
  Since $f \in \cl_M(F_2 \cup F_3)$, we have 
  \[((c_1 \otimes c_2) \otimes c_3)(e) = (c_1 \otimes c_2)(e) = c_2 (f) = (c_1 \otimes (c_2 \otimes c_3))(e). \]  
  Finally, suppose that $e \in \cl_M(F_1 \cup F_2 \cup F_3) - \cl_M(F_1 \cup F_2)$. 
  Let $f' \in \cl_M(F_2 \cup F_3)$ be parallel to $e$ in $M \con F_1$. 
  If we have $f' \in \cl_M(F_2)$ we would have $e \in \cl_{M \con F_1}(F_2) = \cl_M(F_1 \cup F_2)$, a contradiction. 
  So $f' \in \cl_M(F_2 \cup F_3) - \cl_M(F_2)$. Let $f \in F_3$ be parallel to $f'$ in $M \con F_2$. 

  By definition, we have $(c_1 \otimes (c_2 \otimes c_3)) (f') = (c_2 \otimes c_3) (f') = c_3(f)$. 
  On the other hand, since $e$ and $f$ are parallel in $M \con F_1$, while
  $f$ and $f'$ are parallel in $M \con F_2$, 
  all three are parallel in $M \con \cl_M(F_1 \cup F_2)$, which gives
  $((c_1 \otimes c_2) \otimes c_3) (e) = c_3(f) = (c_1 \otimes (c_2 \otimes c_3))(e)$,
  as required. 
\end{proof}

It is easy to see, on the other hand, that lift-joins are not in general commutative, 
so the order of the indices is still important in an expression $\bigotimes_i c_i$. 
Note also that for such an expression to be well-defined, each $c_i$ except possibly $c_1$ must be simple. 

\begin{lemma}\label{ljrainbow}
  Let $F_1$ and $F_2$ be skew flats of a modular matroid $M$, 
  and $c_1,c_2$ be colourings of $F_1,F_2$ respectively so that $c_2$ is simple. 
  If $(M | \cl_M(F_1 \cup F_2); c_1 \otimes c_2)$ has a rainbow triangle, then 
  so does one of $(M | F_1; c_1)$ or $(M | F_2; c_2)$.
\end{lemma}
\begin{proof}
  We may assume that $\cl_M(F_1 \cup F_2) = E(M)$ and that $r_M(F_2) > 0$. 
  Suppose that $T$ is a rainbow triangle in $(M; c_1 \otimes c_2)$.
  
  If $\sqcap_M(F_1,T) = 2$, then $T \sse F_1$, so $T$ is a rainbow triangle of $(M | F_1; c_1)$.
  If $\sqcap_M(T, F_1) = 1$, then let $e$ be a nonloop in $\cl_M(T) \cap F_1$, 
  and let $x \in T - \cl_M(\{e\})$. Since $T \subseteq \cl_M(\{e,x\})$ and 
   $T \not\subseteq F_1$, we know that $x \notin F_1$. 
  Now $T - \cl_M(\{e\}) \subseteq \cl_M(\{e, x\}) - F_1 \subseteq \cl_M(F_1 \cup \{x\}) - F_1$. 
  But by Lemma~\ref{liftjoinbasic}(\ref{ljdecomposer}), the set $\cl_M(F_1 \cup \{x\}) - F_1$ 
  is therefore monochromatic and contains at least two elements of $T$, a contradiction. 

  Suppose, therefore, that $\sqcap_M(F_1,T) = 0$, so $T$ is skew to $F_1$,
  and $(M \con F_1) | T = M | T$.  
  Each $e \in T$ is parallel to some element of $F_2$ in $M \con F_1$. 
  Let $T' \subseteq T$ be a triple containing an element parallel in $M \con F_1$ to each element of $T$. 
  Since $F_1$ and $F_2$ are skew, we have $(M \con F_1) | F_2 = M | F_2$. 
  By the choice of $T'$ we have 
  \[M | T' = (M \con F_1) | T' \cong (M \con F_1) | T = M | T,\]
  and thus $T'$ is a triangle of $M$. 
  But also by the choice of $T'$ and Lemma~\ref{liftjoinbasic}(\ref{ljdecomposer}), 
  we see that $(c_1 \otimes c_2) (T) = c_2(T')$, and so $T'$ is a rainbow triangle of $(M | F_2; c_2)$.
\end{proof}

Given a set $X$ of elements of a matroid $M$, write $\omega_M(X)$ for the rank of the 
largest flat of $M$ that is contained in $X$. (In the case where $M \cong \PG(n-1,2)$, 
this essentially agrees with the definition of $\omega$ given earlier for subsets of $\bF_2^n$.)
The following lemma gives an expression for $\omega_M(X)$, where $X$ is a colour class 
in a lift-join of colourings. 

\begin{lemma}\label{ljomega}
  Let $F_1, \dotsc, F_t$ be mutually skew flats of a connected modular 
  coloured matroid  $(M; c)$, such that $c = \bigotimes_{i=1}^t (c | F_i)$. 
  If $X$ is a colour class of $c$,
  then $\omega_M(X) = \sum_{i = 1}^t \omega_M(F_i \cap X)$. 
\end{lemma}
\begin{proof}
    By induction and associativity, it suffices to show this for $t = 2$. 
    Suppose, therefore, that $c = (c | F_1) \ls (c | F_2)$ for complementary 
    flats $F_1, F_2$ of $M$, and let $X = c^{-1}(\kappa)$ for some $\kappa \in [c]$.

    For each $i$, let $K_i \subseteq F_i \cap X$ and $r_M(K_i) = \omega_M(F_1 \cap X)$.
    We clearly have $c(K_1) = \{\kappa\},$ and for each $e \in \cl_M(K_1 \cup K_2) - K_1$,
    the point $f \in K_2 \cap \cl_M(K \cup \{e\})$ has $c(f) = \kappa$ and $c(f) = c(e)$, 
    so we have $c(\cl_M(K_1 \cup K_2)) = \{\kappa\}$ and thus 
    $\omega_M(X) \ge r_M(K_1 \cup K_2) = r_M(K_1) + r_M(K_2) = \sum_{i=1}^2 \omega_M(F_i \cap X)$.

    For the other direction, let $K$ be a flat of $M$ with $K \subseteq X$ 
    and $r_M(K) = \omega_M(X)$. Let $K_1 = K \cap F_1$, and let $K_2$ be a complement of $K_1$ in $M | K$.
    By construction, we have $\omega_M(X) = r_M(K_1) + r_M(K_2)$, and $r_M(K_1) \le \omega_M(K \cap F_1)$, 
    so it suffices to show that $r_M(K_2) \le \omega_M(F_2 \cap X)$.
    
    Since $K_2$ and $F_1$ are disjoint, modularity implies that they are skew in $M$, 
    so $(M \con F_1) | K_2 = M| K_2$. By modularity, for each $e \in K_2$, the flat $\cl_M(F_1 \cup \{e\})$
    intersects $F_2$ in a unique element $\varphi(e)$, and $e$ and $\varphi(e)$ are parallel in $M \con F_1$. 
    Let $K_2' = \varphi(K_2) \subseteq F_2$. Since $K_2' \subseteq F_2$ and $F_2$ is skew to $F_1$, 
    we have $M | K_2 = (M \con F_1) | K_2 \cong (M \con F_1) | K_2' = M | K_2'$, where the isomorphism is given by 
    $\varphi$. 
    So $K_2'$ is a flat of $M$ of rank equal to $r_M(K_2)$, and since $e \in \cl_M(F_1 \cup \varphi(e))$ 
    for each $e$, we have $c(e) = c(\varphi(e))$ and thus $c(K_2') = c(K_2) \subseteq \{\kappa\}$. 
    It follows that $r_M(K_2) = r_M(K_2') \le \omega_M(F_2 \cap X)$, as required.
\end{proof}

Finally, we can state our main structure theorem, 
which roughly asserts that rainbow-triangle-free coloured matroids are precisely 
those that can be obtained by gluing simple two-coloured flats together using lift-joins. 
The only technical exception is the first summand, which can be non-simple, 
or a rank-one flat that uses more than two colours.

\begin{theorem}\label{mainstructure}
  Let $(M; c)$ be a connected, modular coloured matroid. The following are equivalent: 
  \begin{enumerate}[(i)]
    \item\label{rtf} $(M; c)$ is rainbow-triangle-free;
    \item\label{rcf} $(M; c)$ has no rainbow circuit of size $3$ or more;
    \item\label{rtfdecomp} there are mutually skew nonempty flats $F_i, \dotsc, F_t$ of $M$
      such that $c = \bigotimes_{i=1}^t (c | F_i)$,
      the colouring $c | F_i$ is simple with $|c(F_i)| \le 2$ for all $i \ge 2$, 
      and either $r_M(F_1) = 1$ or $|c(F_1)| \le 2$.
  \end{enumerate}
\end{theorem}
\begin{proof}
  The equivalence of (\ref{rtf}) and (\ref{rcf}) follows from Lemma~\ref{rainbowcircuit}, 
  so it suffices to show that (\ref{rtf}) and (\ref{rtfdecomp}) are equivalent. 
  
  First suppose that there are flats $F_i$ such that (\ref{rtfdecomp}) holds. 
  This implies that $c_i:=c | F_i$ is rainbow-triangle-free for each $i$. 
  It follows by induction and Lemma~\ref{ljrainbow} that 
  $(M; c) = (M; \bigotimes_{i=1}^t c_i)$ is rainbow-triangle-free.
  

  Conversely, suppose that $(M; c)$ is rainbow-triangle-free. 
  If $r(M) = 0$, then we can take $t=0$. 
  If $r(M) = 1$ or $|c| \le 2$, then it is easy to check that we can take $t = 1$ 
  and $F_1 = E(M)$.
  We may assume, therefore, that $|c| \ge 3$ and $r(M) \ge 2$, 
  and that (\ref{rtf}) implies (\ref{rtfdecomp}) for all matroids of smaller rank than $M$. 
  
  By Theorem~\ref{main1}, there is a nontrivial decomposer $F^1$ of $(M; c)$. 
  Let $F^2$ be a complement of $F^1$ in $M$.
  Since $F^1$ is a decomposer, we have that $r_M(F^2) > 0$ 
  and the restriction $c | F^2$ is a simple colouring. 
  
  Using the fact that $F^1$ and $F^2$ are proper flats and induction, 
  for each $i \in [2]$, there are flats $F_1^i, \dotsc, F_{t_i}^i$ 
  satisfying (\ref{rtfdecomp}) in the coloured matroid $(M;c) | F^i$. 
  For each $i \in [t_1]$, let $c^1_i:=c |F^1_i$ and for each $j \in [t_2]$, let $c^2_j=c | F^2_j$.  By Lemma~\ref{liftjoiniff}, 
  $c = (c | F^1) \otimes (c | F^2) = (\bigotimes_{i=1}^{t_1} c^1_i) \otimes (\bigotimes_{j=1}^{t_2} c_j^2)$. 

  We now show that $F_1^1, \dotsc, F_{t_1}^1, F_1^2, \dotsc, F_{t_2}^2$ is a sequence of 
  flats satisfying (\ref{rtfdecomp}) for $M$. 
  This is almost completely straightforward, since $\otimes$ is associative 
  and the flats in the collection $\{F_i^1\} \cup \{F_j^2\}$ are mutually skew. 
  All that remains is to show that the colouring $c | F^2_1$ is simple and that $|c(F^2_1)| \le 2$. 
  The simplicity follows from the fact that $c | F^2_1$ is a restriction of the simple colouring $c | F^2$.
  If $|c(F^2_1)| > 2$ then since $F^2_1, \dotsc, F^2_{t_2}$ satisfy (\ref{rtfdecomp}), 
  we know that $r_M(F^2_1) = 1$, but then simplicity gives $|c(F^2_1)| =1$, a contradiction. 
\end{proof}

Note that any collection of flats satisfying (\ref{rtfdecomp}) must have spanning
union in $M$, since otherwise $\bigotimes_i (c | F_i)$ would not be defined on all of $M$.

When the colouring $c$ is itself simple, all the $c_i$ must be simple, 
so we can guarantee that $|c_i| \le 2$ for all $i$, giving a less technical version of Theorem~\ref{mainstructure}. 
We remark that all the other equivalent conditions in Theorem~\ref{easyequiv} could also be added to this list. 

\begin{theorem}\label{mainstructuresimple}
  Let $(M; c)$ be a connected, modular coloured matroid for which $c$ is simple. 
  The following are equivalent: 
  \begin{enumerate}[(i)]
    \item $(M; c)$ is rainbow-triangle-free;
    \item there are mutually skew nonempty flats 
      $F_1, \dotsc, F_t$ of $M$ such that
      $c = \bigotimes_{i=1}^t (c | F_i)$, and $|c(F_i)| \le 2$ for each $i$. 
  \end{enumerate}
\end{theorem}

In particular, this applies whenever $M$ is the matroid of a 
finite-dimensional Desarguesian projective geometry over a finite field. 
The statement of Theorem~\ref{mainstructuresimple} is easy to extend meaningfully 
to geometries over infinite fields, 
but in fact the extension is false in general, for interesting reasons; 
see Section~\ref{infsection} for a discussion. 

Finally, we formulate the rank-three case of the above
as a theorem about projective planes, which are equivalent to simple, 
modular, connected rank-three matroids. 
The case of Desarguesian planes was independently proved by Carter and Vogt [\ref{cartervogt}] 
and by Hales and Straus [\ref{hs82}]. 
The proof, using Theorem~\ref{mainstructuresimple}, is straightforward. 

\begin{theorem}\label{mainplane}
  Let $c$ be a colouring of the points of a projective plane $P$. 
  Then every line of $P$ contains points of at most two colours 
  if and only if either
  \begin{enumerate}[1.]
    \item\label{twocolplane} $|c| \le 2$,
    \item\label{threecol} there are colours $\alpha, \beta_1, \beta_2$ so that 
      $[c] = \{\alpha, \beta_1, \beta_2\}$, and either 
      \begin{enumerate}[i.]
        \item\label{goodpt} there is a point $x$ of $P$
          so that $c(x) = \alpha$, and each line $L$ through $x$
          satisfies $c(L-\{x\}) = \{\beta_1\}$ or $c(L-\{x\}) = \{\beta_2\}$, or 
        \item\label{goodline} there is a line $L$ of $P$ such that $c(P-L) = \{\alpha\}$ 
          and $c(L) = \{\beta_1, \beta_2\}$, and $L$ contains at least two points of each colour. 
      \end{enumerate}
  \end{enumerate}
\end{theorem}

\subsection*{Kemperman's theorem}

In the case where $M$ is a Desarguesian projective geometry 
and $|c| = 3$, 
Theorem~\ref{mainstructuresimple} can almost certainly be derived from Kemperman's theorem 
on small sumsets in abelian groups. 
Though the original version of the theorem in [\ref{kemperman}] does not use this terminology, 
we phrase our discussion in terms of the shorter proof in [\ref{bdm}], which uses the convenient concept of a 
\emph{critical trio}, defined to be a triple $(A,B,C)$ of subsets of an additive abelian group $G$ for which
$|A| + |B| + |C| > |G|$ and $0 \notin A + B + C$. 

To apply this to Theorem~\ref{mainstructuresimple} with $|c| = 3$ and $M = \PG(n-1,q)$, 
let $A_1,A_2,A_3$ be the colour classes of a rainbow-triangle-free
$3$-colouring of $\PG(n-1,q)$. Each element $e$ of $M$ is by definition a $1$-dimensional 
subspace of $\bF_q^n$. For each $i \in [3]$, let $X_i = \bigcup_{e \in A_i} e \subseteq \bF_q^n$. 
By construction, the sets $X_i$ pairwise intersect only in the zero vector, 
and satisfy $|X_1| + |X_2| + |X_3| = q^n+2$. 
If there exist vectors $x_1,x_2,x_3$ with $x_i \in X_i$ and $x_3 \ne 0$ while $x_1 + x_2 + x_3 = 0$, 
then either some $x_i$ is zero and the other two are parallel and nonzero (which does not occur by the construction of the $X_i$)
or all three are nonzero, which does not occur since the original colouring is rainbow-triangle-free. 
It follows, therefore, that $0 \notin X_1 + X_2 + (X_3 \del \{0\})$ and $|X_1| + |X_2| + |X_3 \del \{0\}| = q^n+1$. 
Therefore $(X_1,X_2,X_3 \del \{0\})$ is a critical trio in the additive group $\bF_q^n$. 

Kemperman's theorem as formulated in [\ref{bdm}, Theorem 4.5] gives an exact characterization of all critical trios, 
in terms of so-called `beats', `chords', `continuations' and a chain of subgroups. 
We do not state the result here, since its statement requires a plethora of definitions, 
but the correspondence with Theorem~\ref{mainstructuresimple} is clear -- 
`impure beats' correspond to lift-joins, and the chain of subgroups correspond to the subspaces spanned by 
unions of initial segments of the $F_i$. Deriving the three-coloured case of Theorem~\ref{mainstructuresimple}
from Kemperman's theorem is likely a matter of little more than reconciling 
distinct sets of terminologies. 

With more than three colours, the simplicity of the correspondence breaks down. 
While one could possibly use an inductive strategy where distinct colours 
are merged to find a proof of Theorem~\ref{mainstructuresimple} using Kemperman's theorem, 
such a proof could easily be longer than ours. 

\section{Targets}


Before applying Theorem~\ref{mainstructuresimple} to derive our main results on the geometric \erdos-Hajnal conjecture, 
we define a special type of colouring. 
This is a version of a definition originally given in [\ref{nn}] in the language of induced submatroids, 
which is essentially the two-coloured case.

Call a coloured matroid $(M; c)$ a \emph{target} if there are flats 
$\cl_M(\es) = F_0 \subseteq F_1 \subseteq \dotsc \subseteq F_k = E(M)$ of $M$
such that every nonempty set of the form $F_0$ or $F_{i+1} - F_i$ is monochromatic. 

Given a target $(M;c)$ with $r(M) \ge 1$, it can be convenient to choose 
the flats $F_0 \subseteq \dotsc \subseteq F_k$ with $k$ minimal. 
Given such a choice, $F_{i+1}-F_{i-1}$ must contain 
at least two colours for all $1 \le i < k$, since otherwise we could remove $F_i$ 
from the list. It follows that there is always a choice of the $F_i$ so that the 
containments $F_i \subseteq F_{i+1}$ are all strict, and for all $i$, 
the sets $F_i - F_{i-1}$ and $F_{i+1} - F_i$ are monochromatic with different colours. 

This section will establish some basic properties of targets, 
as well as prove two theorems that classify targets in terms of excluded subconfigurations.
We omit the trivial proofs of the first two lemmas. 
\begin{lemma}\label{restricttarget}
  If $(M; c)$ is a target and $F$ is a flat of $M$, then $(M; c) | F$ is a target. 
\end{lemma}

\begin{lemma}\label{targetcomp}
  If $(M; c)$ is a target and $d$ is a function whose domain contains $[c]$, 
  then $(M; d \circ c)$ is a target.
\end{lemma}

\begin{lemma}\label{targetifflj}
  A connected, modular coloured matroid $(M;c)$ is a target if and only if $c$ is a lift-join
  of a family of monochromatic flats of $M$. 
\end{lemma}
\begin{proof}
  Suppose that $\es = F_0 \subseteq \dotsc F_k = E(M)$ are flats certifying that $(M;c)$ 
  is a target. For each $1 \le i \le k$, since $F_i - F_{i-1}$ is monochromatic, 
  the flat $F_{i-1}$ is a decomposer of $M | F_i$, so by Lemma~\ref{liftjoiniff},
  there is a complement $K_i$ of $F_{i-1}$ in $M | F_i$ for which $c | F_i = (c | F_{i-1}) \ls (c | K_i)$.
  It follows easily by induction on $t$ that $c | F_t = \bigotimes_{i=1}^t (c | K_i)$ for all $1 \le t \le k$.
  Setting $k = t$ gives that $c = \bigotimes_{i=1}^k (c | K_i)$, as required. 

  Conversely, suppose that $c = \bigotimes_{i=1}^k (c | K_i)$, where 
  $K_1, \dotsc, K_k$ are mutually skew monochromatic flats of $M$ with spanning union.
  It is immediate that the collection of flats $F_i : 0 \le i \le k$ defined by 
  $F_i = \cl\left(\bigcup_{j = 1}^i K_j \right)$ certifies that $(M; c)$ is a target. 
\end{proof}

The next lemma relates targets to decomposers. 
\begin{lemma}\label{targetdecomposer}
  Let $F$ be a decomposer of a modular coloured matroid $(M; c)$, 
  and $F'$ be a complement of $F$ in $M$. 
  If $(M;c) | F$ and $(M; c) | F'$ are targets, then $(M; c)$ is a target. 
\end{lemma}
\begin{proof}
    
    By hypothesis and Lemma~\ref{targetifflj}, we know that $c | F$ is a 
    lift-join of a family of monochromatic flats of $M | F$, 
    and likewise for $c | F'$ and $M | F'$. 
    By Lemma~\ref{liftjoiniff}, we have $c = (c | F) \ls (c | F')$.
    The result now follows from associativity and Lemma~\ref{targetifflj} applied to $(M;  c)$.
\end{proof}

\begin{lemma}\label{linetarget}
  A rank-two coloured matroid $(M; c)$ is a target if and only if $M$ has a point $P$
  for which $P$ and $E(M) - P$ are monochromatic. 
\end{lemma}
\begin{proof}
  If $M$ has such a point $P$, the flats $\cl_M(\es) \subseteq P \subseteq E(M)$ certify that $(M ;c)$ is a target. 
  If $M$ is a target, let $F_0, \dotsc, F_k$ be flats of $M$ such that $1 \le k \le 2$ and 
  $\cl_M(\es) = F_0 \subset \dotsc \subset F_k = E(M)$ while each set $F_{i+1} - F_i$ is monochromatic.
  If $k = 1$ then $M$ is monochromatic and any point $P$ satisfies the lemma. 
  If $k = 2$, then $F_1$ is a point satisfying the lemma.  
\end{proof}

\begin{lemma}\label{decomposetarget}
  If $(M; c)$ is a modular target for which $|c| \ge 2$, 
  then there are distinct colours $\kappa,\lambda$ and a nontrivial decomposer $F$ of $M$
  such that $E(M)-F$ is monochromatic with colour $\kappa$, 
  and $M|F$ is a target that has a monochromatic basis of colour $\lambda$. 
\end{lemma}
\begin{proof}
  Let $k$ be minimal so that there are flats $F_0, \dotsc, F_k$ of $M$ with 
  $\cl_M(\es) = F_0 \subset \dotsc \subset F_k = E(M)$ and each $F_{i+1} - F_i$ monochromatic. 
  Since $|c| \ge 2$ we have $k \ge 2$. 
  Let $\{\kappa\} = c(E(M)-F_{k-1})$ and $\{\lambda\} = c(F_{k-1}-F_{k-2})$. 
  If $\kappa = \lambda$, we could remove $F_{k-1}$ from the sequence, which contradicts minimality,
  so $\kappa \ne \lambda$. 
  Since $F_{k-2}$ is a proper subflat of $F_{k-1}$,
  the set $F_{k-1} - F_{k-2}$ is codependent in $M | F_{k-1}$, 
  so it contains a cocircuit, which by Lemma~\ref{modularround} is spanning. 
  It therefore contains a monochromatic basis of $F_{k-1}$ with colour $\lambda$, 
  so $F = F_{k-1}$ satisfies the lemma.
\end{proof}

\begin{lemma}\label{minimalnottarget}
  Let $(M;c)$ be a connected, modular coloured matroid with $r(M) \ge 3$
  that is not a target, but for which $(M;c)|H$ is a target for every hyperplane $H$ of $M$. Then 
  \begin{enumerate}[(i)]
    \item\label{subtarget} $(M; c)| F$ is a target for every proper flat $F$ of $M$, 
    \item\label{nodecomp} $M$ has no nontrivial decomposer, and
    \item\label{twocolour} $|c|=2$.
  \end{enumerate}
\end{lemma}
\begin{proof}
  Since every proper subflat is contained in a hyperplane, (\ref{subtarget}) is immediate by Lemma~\ref{restricttarget}. 

  To see (\ref{nodecomp}), let  $F$ be a nontrivial decomposer of $M$, and 
  $F'$ be a complement of $F$ in $M$.
  We have $r_M(F') = r(M) - r_M(F) < r(M)$, so both $F$ and $F'$ are proper subflats and are therefore
  targets by (\ref{subtarget}). It follows from Lemma~\ref{targetdecomposer} that $M$ is a target, a contradiction.

  Since $M$ is not a target, we have $|c| \ge 2$. 
  If $M$ has a rainbow triangle $T$, then $(M; c) | \cl_M(T)$ is not a target by Lemma~\ref{linetarget}, 
  contrary to~(\ref{subtarget}). So $M$ has no rainbow triangle. 
  If $|c| \ge 3$, it thus follows from Theorem~\ref{main1} that $(M;c)$ has a nontrivial decomposer $F$, 
  contradicting (\ref{nodecomp}). 
\end{proof}

This first theorem shows that a modular coloured matroid whose planes are targets is a target. 
In the special case of two-coloured binary projective geometries,
this was essentially proved in [\ref{nn}]. 

\begin{theorem}\label{targetiffplane}
  Let $(M ;c)$ be a connected modular coloured matroid with $r(M) \ge 3$.
  Then $(M ;c)$ is a target if and only if $(M; c) | P$ is a target for every plane $P$ of $M$. 
\end{theorem}
\begin{proof}
  Let $(M; c)$ be a counterexample of minimum rank, 
  so $(M;c)$ is not a target, but restricting to any plane gives a target, and hence $r(M) \ge 4$. 
  By this minimality, the coloured matroid $(M;c) |H$ is a target
  for each hyperplane $H$ of $M$, so $(M;c)$ satisfies the hypotheses of Lemma~\ref{minimalnottarget}. 
  
  By Lemma~\ref{minimalnottarget}(\ref{twocolour}), we have $|c| = 2$.
  Let $H$ be a hyperplane of $M$ for which $|c(H)| = 2$, noting that $(M;c)|H$ is a target.
  By the choice of $H$ and Lemma~\ref{decomposetarget}, there is a nontrivial decomposer $F$ 
  of $(M ; c)|H$ and a basis $B$ of $F$ such that $c(B) = \{\lambda\}$ and $c(H-F) = \{\kappa\}$ 
  for distinct colours $\kappa, \lambda$. 

  Since $F$ is not a decomposer of $M$, there exists $e \in E(M)-F$ such that the set $A = \cl_M(F \cup \{e\})-F$
  satisfies $c(A) = \{\kappa,\lambda\}$. 
  Note that $F \cup A$ is a flat, and $F$ is a hyperplane of the matroid $M|(F \cup A)$.
  By Lemma~\ref{basispartition} applied to this matroid with the basis $B$ of $F$, we see that $M$ has a triangle $\{x,y,b\}$
  with $x,y \in \cl_M(F \cup \{e\}) - F$ and $b \in B$, such that $c(x) = \kappa$ and $c(y) = \lambda$. 

  Since $F \cup A$ is not a decomposer of $M$ by Lemma~\ref{minimalnottarget}(\ref{nodecomp}), 
  there exists $z \in E(M) - (F \cup A)$ with $c(z) = \lambda$.
  Let $P = \cl_M(\{x,y,z\})$. Since $z \notin F \cup A$ we have $z \in P- \cl_M(\{y,b\})$,
  so $\{y,b,z\}$ is a monochromatic basis of $P$ with colour $\lambda$. 

  Since $x \in P - H$, the set $H \cup P$ is spanning, so by modularity $r_M(H \cap P) = \sqcap_M(P,H) = 2$, 
  and the flat $H \cap P$ is a line. If $H \cap P \subseteq F$, then $(H \cap P) \cup \{x\}$ is a rank-$3$ subset 
  of $P$ contained in the flat $F \cup A$, and so $P \subseteq F \cup A$, which contradicts $z \in P$. 
  Therefore $H \cap P$ contains at most one element of $F$, and $c((H \cap P) - F) \subseteq c(H-F) = \{\kappa\}$, 
  so (since $M|(H \cap P)$ is connected), the line $H \cap P$ contains at least two elements coloured $\kappa$. 
  Now $x$ is a third element of $P$ coloured $\kappa$ outside this line, so $P$ has a monochromatic basis coloured $\kappa$. 

  We have just seen that $P$ has a monochromatic basis of each colour. 
  But $(M;c)|P$ is a target by Lemma~\ref{minimalnottarget} (\ref{subtarget}), 
  so $P$ has a proper subflat $F_0$ with $P-F_0$ monochromatic, 
  implying that some colour class of $(M;c)|P$ lies in $F_0$ so is nonspanning in $P$. This is a contradiction. 
\end{proof}

With an additional hypothesis that the lines are not too short, we can upgrade the previous theorem 
to be stated in terms of lines rather than planes. It is impossible to drop this extra hypothesis, 
since all rank-$2$ coloured binary projective geometries are targets, but not all coloured binary projective 
geometries are targets. 

The following theorem was also proved for $2$-colourings in [\ref{mo}, Theorem 2].
Our proof is also a little shorter, since it indirectly exploits Theorem~\ref{main1}. 

\begin{theorem}\label{targetiffline}
  Let $(M;c)$ be a connected modular coloured matroid having a line with at least four points. 
  Then $(M;c)$ is a target if and only if $(M;c)|L$ is a target for each line $L$ of $M$.
\end{theorem}
\begin{proof}
  Let $M$ be a minimum-sized counterexample. By Lemma~\ref{restricttarget}, it follows that $M$ is not a target, 
  but $(M;c)|L$ is a target for every line $L$, which implies that $r(M) \ge 3$. 
  By Lemma~\ref{targetiffplane}, there is a plane $P$ of $M$ such that $(M;c) |P$ is not a target. 
  Since $(M;c)|P$ is a counterexample, we must have $P = E(M)$ by minimality. 
  
  By the minimality in the choice of $M$, the coloured matroid $(M;c) |H$ is a target
  for each hyperplane $H$ of $M$, so $(M;c)$ satisfies the hypotheses of Lemma~\ref{minimalnottarget}. 
  By Lemma~\ref{minimalnottarget}(\ref{subtarget}), we see that each restriction to a point
  is a target, so each point is monochromatic. By the minimality in the choice of $M$, 
  it easily follows that $M$ is simple.

  Now $M$ is a simple, connected rank-three modular matroid that has a line with at least four elements, 
  so is a projective plane of order $q \ge 3$. 

  By Lemma~\ref{minimalnottarget}, we have $|c| = 2$; let $c(M) = [2]$, 
  and for each $i$ let $X_i$ be the set of points of colour $i$. 
  Since $(M;c)|L$ is a target for each line $L$, 
  by Lemma~\ref{linetarget} we see the following: 
  \begin{claim}\label{bigline}
    If $i \in [2]$ and $x \ne y \in X_i$, then $|X_i \cap \cl_M(\{x,y\})| \ge q$. 
  \end{claim}

  Fix $i \in [2]$. If $X_i$ is not spanning, then $E(M)-\cl_M(X_i)$ is nonempty 
  and monochromatic, so $\cl_M(X_i)$ is a decomposer of $(M;c)$, 
  contradicting Lemma~\ref{minimalnottarget}(\ref{nodecomp}). 
  So $X_i$ contains a basis $B_i = \{x_1,x_2,x_3\}$. 
  By~\ref{bigline}, the line $L = \cl_M(\{x_1,x_2\})$ satisfies $|L \cap X_i| \ge q$, 
  and for each $y \in L \cap X_i$, the line $L' = \cl_M(\{x_3,y\})$ satisfies $|L' \cap X_i| \ge q$. 
  The lines in the collection $\cL = \{\cl_M(\{x_3,y\}): y \in L \cap X_i\}$ 
  pairwise intersect only at $x_3$,  so
  \[|X_i| \ge 1 + \sum_{L' \in \cL}|(L'-x_3) \cap X_i| \ge 1 + q(q-1) = q^2-q+1.\]
  But this holds for both $X_1$ and $X_2$, so $q^2+q + 1 = |M| = |X_1| + |X_2| \ge 2(q^2-q+1)$. 
  This gives $q^2 + 1 \le 3q$, which contradicts $q \ge 3$. 
\end{proof}

Recall that, for $X \subseteq E(M)$, we write $\omega_M(X)$ for the rank of the largest flat of $M$ contained in $X$. 

\begin{lemma}\label{targetomega}
  Let $(M;c)$ be a target for which $M$ is connected and modular, 
  and for each $\lambda \in [c]$, let $X_\lambda$ be the colour class of $\lambda$. 
  If $L \subseteq [c]$, 
  then $\sum_{\lambda \in L} \omega_M(X_\lambda) = \omega_M( \cup_{\lambda \in L} X_\lambda)$. 
\end{lemma}
\begin{proof}
  Let $F_1, \dotsc, F_s$ be mutually skew monochromatic flats of $M$ with spanning union,
  such that $c = \bigotimes_{i=1}^s (c | F_i)$, as given by Lemma~\ref{targetifflj}.
  Let $\{\lambda_i\} = c(F_i)$, and let $I = \{i \in \{1, \dotsc, s\} : \lambda_i \in L \}$. 
  Then, by Lemma~\ref{ljomega}, we have 
  \[\sum_{\lambda \in L} \omega_M(X_\lambda) = \sum_{\lambda \in L} \sum_{i=1}^s \omega_M(F_i \cap X_{\lambda}) = \sum_{i=1}^s \sum_{\lambda \in L} \omega_M(F_i \cap X_{\lambda}).\]
  Now $F_i \cap X_\lambda$ is equal to $F_i$ if $\lambda = \lambda_i$, and is empty otherwise. 
  It follows that $\sum_{\lambda \in L} \omega_M(F_i \cap \lambda) = r_M(F_i)$ if $i \in I$, 
  and $\sum_{\lambda \in L} \omega_M(F_i \cap \lambda) = 0$ otherwise. Therefore 
  \begin{equation}\label{tolhs}\sum_{\lambda \in L} \omega_M(X_\lambda) = \sum_{i \in I} r_M(F_i).
  \end{equation}

  Now let $d : [c] \to \{0,1\}$ be the colouring with $d(\lambda) = 1$ if and only if $\lambda \in L$, 
  and let $X = (d \circ c)^{-1}(1)$; so $X = \bigcup_{\lambda \in L} X_{\lambda}$ by construction. 
  Since $X$ is a colour class of the colouring $d \circ c = \bigotimes_{i=1}^s ((d \circ c) | F_i)$, Lemma~\ref{ljomega} gives 
  \[\omega_M(X) = \sum_{i=1}^s \omega_M(F_i \cap X).\]
  Now $F_i \cap X$ is equal to $F_i$ if $i \in I$, and is empty otherwise. 
  So in fact $\omega_M(X) = \sum_{i \in I}r_M(F_i)$. Combined with (\ref{tolhs}), 
  this gives the result. 
\end{proof}

We now restate and prove Theorem~\ref{binaryintro}. In the language of this section, 
it states that rainbow-triangle-free coloured binary projective geometries where 
the number of colours is equal to the dimension are all targets, determined by a 
maximal chain of flats. This implies that they are unique up to matroid isomorphism and permutation of colours. 

\begin{theorem}\label{fullbinary}
    Let $q$ be a prime power, and $(M;c)$ be a coloured matroid with no rainbow triangle for which $M \cong \PG(n-1,q)$ and $|c| = n \ge 1$. 
    Then there exist a permutation $\kappa_1, \dotsc, \kappa_n$ of $[c]$ 
    and flats $\es = F_0 \subset F_1 \subset \dotsc F_n = E(M)$ of $M$, 
    such that $r_M(F_i) = i$ and $c(F_i - F_{i-1}) = \{\kappa_i\}$ for each $i$.
\end{theorem}
\begin{proof}
    It is easy to see, along the lines of the proof of Theorem~\ref{targetifflj}, that 
    a coloured matroid $(M;c)$ satisfies the conclusion if and only if $c$ is a lift-join of 
    rank-one monochromatic flats of distinct colours. Thus, it is enough to prove that $M$ has this property. 

    If $n = 2$, then since every triple of points of $M$ is a triangle, 
    the result is straightforward. 
    Suppose inductively that $n \ge 3$, and that the theorem is true for smaller $n$. 
    Since $|c| = n \ge 3$, Theorem~\ref{main1} implies that $(G;c)$ has a nontrivial decomposer $F$. 

    Let $F'$ be a complement of $F$ in $M$. By Lemma~\ref{liftjoiniff} we have $c = (c | F) \ls (c | F')$. 
    By Theorem~\ref{easyequiv}, we have $|c(F)| \le r_M(F)$ and similar for $F'$, so 
    \[r_M(F) + r_M(F') = n = |c| = |c(F)| + |c(F')| - |c(F) \cap c(F')|, \]
    and so $r_M(F) = |c(F)|$ and $r_M(F') = |c(F')|$ while $c(F) \cap c(F') = \es$. 
    By induction, both $c|F$ and $c|F'$ are lift-joins of monochromatic rank-one flats
    of distinct colours. The sets of colours are disjoint, and $c = (c | F) \ls (c | F')$, 
    and so the same is true for $M$, as required. 
\end{proof}

\section{Erd\H{o}s-Hajnal}

This section will derive the lower bound in Theorem~\ref{mainrainbow} from Theorem~\ref{mainstructuresimple}. 
We first prove that a rainbow-triangle-free coloured matroid can be
partioned into two-coloured pieces, in such a way that the partition (viewed as a colouring) 
is a target. This easily implies Theorem~\ref{maintargetintro}, since the `layers' of the
target $(M,C)$ are bicoloured relative to $c$. 

\begin{theorem}\label{targetpaircolouring}
  Let $(M; c)$ be a connected, modular coloured matroid with no rainbow triangle. 
  If $c$ is simple and $|c| \ge 2$, 
  then $M$ has a colouring $C$ for which $[C] \subseteq \binom{[c]}{2}$, 
  such that $(M; C)$ is a target, 
  and $c(e) \in C(e)$ for each $e \in E(M)$. 
\end{theorem}
\begin{proof}
  The result is vacuously true if $r(M) \le 1$, since a simple colouring $c$ of $M$ with $|c| \ge 2$ cannot exist. 
  Suppose, therefore, that $r(M) \ge 2$, and that the result holds for coloured matroids of smaller rank. 

  If $|c| = 2$, then we can take $C$ to assign $[c] \in \binom{[c]}{2}$ to every element of $M$, 
  and $C$ is a $1$-colouring so is a target. Suppose, therefore, that $|c| \ge 3$. 

  By Theorem~\ref{main1}, $M$ has a nontrivial decomposer $F_0$.
  Let $F_1$ be a complement of $F_0$ in $M$.
  Since both $F_0$ and $F_1$ are proper flats, 
  there inductively exist colourings $C_0 \sse \binom{c(F_0)}{2}$ and $C_1 \sse \binom{c(F_1)}{2}$ 
  of $F_0$ and $F_1$ respectively such that $(M | F_0; C_0)$ and $(M | F_1; C_1)$ are targets, 
  and $c (e) \in C_i(e)$ for all $i \in [2]$ and $e \in F_i$. 

  We claim that the colouring $C_0 \otimes C_1$ works for $M$. 
  First, we have $[C_0 \otimes C_1] = [C_0] \cup [C_1] \subseteq \binom{[c(F_0)]}{2} \cup \binom{[c(F_1)]}{2} \subseteq \binom{[c]}{2}$. 
  By Lemma~\ref{targetdecomposer},
  we know that $(M; C_0 \otimes C_1)$ is a target. It remains to show that $e \in (C_0 \otimes C_1)(e)$ for all $e \in E(M)$. 
  
  If $e \in F_0$, then $(C_0 \otimes C_1) (e) = C_0(e) \ni c(e)$,
  as required. If $e \notin F_0$, then let $f \in \cl_M(F_0 \cup \{e\}) \cap F_1$. 
  Since $F_0$ is a decomposer of $(M; c)$, we have $c(e) = c(f)$. 
  Since $F_0$ is a decomposer of $(M; C_0 \otimes C_1)$, 
  we also have $(C_0 \otimes C_1)(e) = (C_0 \otimes C_1)(f) = C_1(f) \ni c(f) = c(e)$, as required.
\end{proof}

Since two-coloured pieces in the above theorem form a target, we can easily find a large subspace 
in one of them. This is analogous to [\ref{fgp}, Corollary 1.5].

\begin{theorem}\label{mainomega}
  Let $(M; c)$ be a connected, modular, rainbow-triangle-free coloured matroid with $c$ simple and $|c| \ge 2$. 
  Then $M$ has a flat $F$ for which $r_M(F) \ge r(M)/ \binom{|c|}{2}$ and $|c(F)| \le 2$. 
\end{theorem}
\begin{proof}
  Let $C$ be a colouring given by Theorem~\ref{targetpaircolouring}, and for each pair $P \in \binom{[c]}{2}$, 
  let $X_P$ be the set of elements of $M$ coloured $P$ by $C$. By the choice of $C$, the sets $X_P$ partition $E(M)$. 
  
  By Lemma~\ref{targetomega} we have $\sum_P \omega_M(P) = r(M)$, so there exists $P \in \binom{[c]}{2}$ 
  for which $\omega_M(X_P) \ge r(M)/ \binom{|c|}{2}$. Now any flat $F \subseteq C_P$ of rank $\omega_M(X_P)$ 
  is two-coloured, so satisfies the theorem. 
\end{proof}

In fact, the techniques in Theorem~\ref{mainomega} routinely generalize to the case where we are looking for a large flat using at most
some prescribed number of colours. In the following result, Theorem~\ref{mainomega} is the case where $\ell = 2$. 
More generally, the result is very similar to the upper bound [\ref{fgp}, Theorem 1.6].

\begin{theorem}\label{fewercolours}
  Let $\ell,k \in \bN$ with $\ell \le k$ and $k \ge 2$,  
  and let $(M; c)$ be a connected, modular coloured matroid for which 
  $c$ is simple and has no rainbow triangle. If $|c| = k$, 
  then $M$ has a flat $F$ such that $|c(F)| \le \ell$ and 
  $r_M(F) \ge \tfrac{\binom{\ell}{2}}{\binom{k}{2}} r(M)$. 
\end{theorem}
\begin{proof}
  We may assume that $\ell \le k$, as otherwise the result is trivial. 
  Let $C : E(M) \to \binom{[c]}{2}$ be a colouring of $M$ as given by Lemma~\ref{targetpaircolouring}, 
  so $(M;C)$ is a target, and $c(e) \in C(e)$ for all $e \in E(M)$. 

  For each $P \in \binom{[c]}{2}$, let $X_P$ be the colour class of $P$ with respect to $C$. 
  Applying Lemma~\ref{targetomega} with $L = \binom{S}{2}$ and later with $L = \binom{[c]}{2}$, 
  we have 
  \begin{align*}
    \sum_{S \in \binom{[c]}{\ell}} \omega_M \left(\cup_{P \in \binom{S}{2}} X_P\right) 
    &= \sum_{S \in \binom{[c]}{\ell}} \sum_{P \in \binom{S}{2}} \omega_M(X_P) \\ 
    &= \sum_{P \in \binom{[c]}{2}}\sum_{P \subseteq S \in \binom{[c]}{\ell}} \omega_M(X_P) \\ 
    &= \tbinom{k-2}{\ell-2} \sum_{P \in \binom{[c]}{2}} \omega_M(X_P) \\ 
    &= \tbinom{k-2}{\ell-2} \omega_M \left(\cup_{P \in \binom{[c]}{2}} X_P\right) \\ 
    &= \tbinom{k-2}{\ell-2} r(M).
  \end{align*}
  By a majority argument, there is therefore some $S \in \binom{[c]}{\ell}$ for which 
  
  \[\omega_M\left(\cup_{P \in \binom{S}{2}} X_P\right) \ge \binom{k}{\ell}^{-1} \binom{k-2}{\ell-2} r(M)  = \tfrac{\binom{k}{2}}{\binom{\ell}{2}} r(M).\] 
  By the choice of the colouring $C$ and the $X_P$, every element in $\cup_{P \in \binom{S}{2}} X_P$
  has a colour in the $\ell$-element set $S$, 
  so $\cup_{P \in \binom{S}{2}} X_P$ contains a flat $F$ of rank at least
  $\tfrac{\binom{k}{2}}{\binom{\ell}{2}} r(M)$ with $c(F) \subseteq S$, as required. 
\end{proof}

We can now prove the lower bound in Theorem~\ref{mainrainbow}. 
\begin{corollary}\label{ehrainbowtri}
  Let $c_0$ be a $3$-colouring of  $\PG(1,2)$ 
  where the three elements have distinct colours. 
  If $n \ge 2$ and $c$ is a $c_0$-free $3$-colouring of $G = \PG(n-1,2)$, 
  then $G$ has a subspace of dimension at least $\tfrac{n}{3}$ containing
  points of at most two colours. 
\end{corollary}
\begin{proof}
  In the connected modular coloured matroid $(M;c)$ corresponding to $G$ and $c$,
  a rainbow triangle would give a copy of $c_0$, so $(M; c)$ is rainbow-triangle-free.
  By Theorem~\ref{mainomega}, there is a flat $F$ of $M$ with 
  $r_M(F) \ge r(M)/\binom{3}{2} = n/3$ and $|c(F)| \le 2$. 
  This flat gives the required subspace. 
\end{proof}

\section{Lower Bounds}

We now prove using standard probabilistic arguments that the bound in Theorem~\ref{fewercolours} 
(and hence Theorems~\ref{ehrainbowtri} and~\ref{mainomega})
is best-possible up to an additive logarithmic error term.  
We first need a simple lemma about random red-blue coloured geometries. 

\begin{lemma}\label{randomub}
  Let $q$ be a prime power, let $n \in \bN$, and $\kappa, \lambda$ be colours. 
  If $(M; c)$ is a coloured matroid for which 
  $M \cong \PG(n-1,q)$, and $c$ independently assigns each element of $M$ colour 
  $\kappa$ or $\lambda$ with probability $\tfrac{1}{2}$, 
  then for all $\epsilon > 0$, we asymptotically almost surely have $\omega_M(X) \le (1 + \epsilon) \log_q(n)$ 
  for each colour class $X$. 
\end{lemma}
\begin{proof}
  For each $k \in \bN$, let $f(k) = \tfrac{q^k-1}{q-1}$. 
  Let $i \in \{\kappa, \lambda\}$, and $X = c^{-1}(i)$ be the colour class of $i$. 
  Let $t \in \bN$. The number of rank-$t$ flats of $M$ is $\gbinom{n}{t}{q}$. 
  Each such flat has $f(t)$ elements, 
  so for each $i \in \{\kappa, \lambda\}$,
  the probability that such a flat is monochromatic with colour $i$ is $2^{-f(t)}$, 
  so by the union bound, the probability that $(M;c)$ has a monochromatic flat with 
  colour $i$ is at most $\gbinom{n}{t}{q} 2^{-f(t)}$. Now
  \begin{align*}
    \log_2\left(\gbinom{n}{t}{q} 2^{-f(t)}\right) &= 
      \sum_{i=1}^t \log_2\left(\frac{q^{n-i+1}-1}{q^i-1}\right) - f(t) \\ 
      &\le \sum_{i=1}^t \log_2(q^{n}) - f(t) \\ 
      &\le t n \log_2 q - q^{t-1}.
  \end{align*}
  If $t \ge (1+\epsilon) \log_q n $, then 
  $tn \log_2 q - q^{t-1} \le (1 + \epsilon) n \log_2 n - q^{-1}n^{1 + \epsilon}$. 
  This expression tends to $-\infty$ as $n$ grows, 
  whence we see that almost surely $M$ has no monochromatic flat of rank at least $t$, 
  giving $\omega_M(X) \le t$.
\end{proof}

The next result shows that the bound in Theorem~\ref{fewercolours} cannot be improved by more than 
$2 k\log_q n$. The true best-possible bound will depend on geometric Ramsey numbers, so is currently out of reach. 

\begin{theorem}
  Let $q$ be a prime power, let $\ell, k \in \bN$ with $k \ge 2$ and $1 \le \ell \le k$. 
  For all sufficiently large $n$, 
  there is a rainbow-triangle-free coloured matroid $(M; c)$ with $|c| = k$, 
  such that $M \cong \PG(n-1,q)$, and every flat $F$ of $M$ with $|c(F)| \le \ell$ 
  satisfies $r_M(F) \le \frac{\binom{\ell}{2}}{\binom{k}{2}}n  + 2 k \log_q (n)$. 
\end{theorem}
\begin{proof}
  Let $t = \binom{k}{2}$, and let $S_1, \dotsc, S_t$ be an enumeration of $\binom{[c]}{2}$. 
  Let $\epsilon = \tfrac{1}{2(k-\ell)}$,
  and choose $n$ large enough so that $n \ge q^{\ell}$ and so that Lemma~\ref{randomub}
  applies for $\epsilon, q$, and every matroid of rank at least $\floor{\frac{n}{t}}$. 

  Let $G_1, \dotsc, G_t$ be a collection of mutually skew flats of $M$ with spanning union, 
  each having rank in $\{\floor{n/t}, \ceil{n/t}\}$. For each $i \in [t]$, let $c_i$ be a 
  colouring of $G_i$ with the two colours in $S_i$ for which there is no monochromatic 
  flat of dimension larger than $\left(1 + \epsilon\right) \log_q (r_M(G_i))$, 
  as given by Lemma~\ref{randomub}. We show that the colouring $c = \bigotimes_{i=1}^t c_i$ 
  of $M$ satisfies the theorem.

  \begin{claim}
    For each $T \in \binom{[c]}{\ell}$, we have $\omega_M(c^{-1}(T)) \le \tfrac{\binom{\ell}{2}}{t} n + \ell k \log_q(n)$. 
  \end{claim}
  \begin{subproof}
    Let $d : [c] \to \{0,1\}$ be the indicator function of $T$. 
    For each $b \in \{0,1,2\}$, let $I_b = \{i \in [t] : |P_i \cap T| = b\}$, 
    so $[t] = I_0 \cup I_1 \cup I_2$, and $|I_b| = \binom{\ell}{b} \binom{k-\ell}{2-b}$. 
    Let $X = c^{-1}(T) = (d \circ c)^{-1}(1)$. 
    Since $d \circ c = \bigotimes_{i=1}^t (d \circ c_i)$, 
    Lemma~\ref{ljomega} implies that 
    \begin{align*}
      \omega_M(X) = \sum_{i=1}^t \omega_M(G_i \cap X) = 
        \sum_{b = 0}^2 \sum_{i \in I_b} \omega_M(G_i \cap X)
    \end{align*}
    For $i \in I_0$, the elements of $G_i$ have colour outside $T$, so $G_i \cap X = \es$. 
    For $i \in I_1$, only one colour assigned to the elements of $G_i$ lies in $T$, 
    so $\omega_M(G_i \cap X)$ is the rank of some monochromatic flat of $c_i$, 
    which (by the choice of $c_i$) is at most 
    $(1 + \epsilon) \log_q(r_M(G_i)) \le (1 + \epsilon) \log_q(n)$. 
    For $i \in I_2$, we have $X \subseteq G_i$, 
    so $\omega_M(G_i \cap X) = r_M(G_i) \le \tfrac{n}{t} + 1$. 
    Combining these estimates gives 
    \begin{align*} \omega_M(X) &\le |I_2| (\tfrac{n}{t} + 1) + |I_1| (1+\epsilon) \log_q (n) \\ 
      &=  \tbinom{\ell}{2} \left(\tfrac{n}{t}+1\right) + \ell(k-\ell) (1 + \epsilon) \log_q(n) \\ 
      &= \tfrac{\binom{\ell}{2}}{t} n + \ell ((k - \ell) (1 + \tfrac{1}{2(k-\ell)}) \log_q(n) + \tfrac{1}{2}(\ell-1)) \\ 
      &\le \tfrac{\binom{\ell}{2}}{t} n + \ell k \log_q(n), 
    \end{align*}
    where the last inequality uses $1 \le \ell \le \log_q(n)$ and the choice of $\epsilon$. 
  \end{subproof}
  Since $t = \binom{k}{2}$,
  we have $r_M(F) \le \tfrac{\binom{\ell}{2}}{\binom{k}{2}} n + \ell k \log_q(n)$ for every 
  flat $F$, as required.  
\end{proof}

We state two corollaries showing that the bounds in Theorems~\ref{mainomega} and~\ref{ehrainbowtri}
are nearly sharp, mirroring the lower and upper bounds given in [\ref{fgp}, Theorem 1.6]. 
The second completes the proof of Theorem~\ref{mainrainbow}. 

\begin{corollary}
  Let $q$ be a prime power, and $k \in \bN$. For all sufficiently large $n$, 
  there is a rainbow-triangle-free coloured matroid $(M; c)$ with $|c| = k$, 
  such that $M \cong \PG(n-1,q)$, and every flat $F$ of $M$ with $|c(F)| \le 2$ 
  satisfies $r_M(F) \le n / \binom{c}{2} + 2 k \log_q (n)$.   
\end{corollary}

\begin{corollary}
  For all sufficiently large $n$, there is a rainbow-triangle-free $3$-colouring of $\PG(n-1,2)$, 
  such that every bicoloured subspace has dimension at most $\tfrac{n}{3} + 6 \log_2 (n)$.   
\end{corollary}
  
\section {Ternary Lines}

We now prove Theorem~\ref{mainchar3}, which we rephrase more concretely
and in terms of two-colourings rather than induced subgeometries. 
The key observation is that ternary projective geometries containing no line
with exactly two points of each colour are targets by Theorem~\ref{targetiffline}.

\begin{theorem}
  Let $c_0$ be a $2$-colouring of the four points of $\PG(1,3)$, containing 
  exactly two points of each colour. 
  If $n \ge 2$, and $c$ is a $c_0$-free $2$-colouring of $G = \PG(n-1,3)$, 
  then $G$ has a monochromatic subspace of dimension at least $\ceil{\tfrac{n}{2}}$.
  Moreover, the value of $\ceil{\tfrac{n}{2}}$ is best-possible. 
\end{theorem}
\begin{proof}
  If some line $L$ of $M \cong \PG(n-1,3)$ contains exactly two points of each colour, then 
  (since any ordered pair of distinct points of $\PG(1,3)$ is related by 
  a linear isomorphism to any other) there is a copy of $c_0$ inside $G$. 
  
  It follows that each line of $M$ contains at most one point of some colour, 
  and therefore $(M; c)|L$ is a target for every line $L$. By Theorem~\ref{targetiffline},
  we see that $(M;c)$ itself is a target. If $X_1$ and $X_2$ are the colour classes of $c$, 
  we have $\omega_M(X_1) + \omega_M(X_2) = n$, so $\omega(X_i) \ge \tfrac{n}{2}$ for some $i$, 
  so some $X_i$ contains the required monochromatic subspace. 

  To see that $\ceil{\tfrac{n}{2}}$ is best possible, consider a $2$-colouring $c$ of the 
  matroid $M = \PG(n-1,3)$
  in which one colour class is a rank-$\ceil{\tfrac{n}{2}}$-dimensional flat $F$. 
  Every monochromatic flat $F'$ of $M$ with rank greater than $\tfrac{n}{2}$ must intersect $F$
  due to a dimension argument, so must be contained in $F$ by the construction of $c$; 
  but then $\ceil{\tfrac{n}{2}} < r_M(F') \le r_M(F) = \ceil{\tfrac{n}{2}}$, a contradiction. 
\end{proof}

\section{Claws}

The following was proved in [\ref{bkknp}], Lemma 4.5. 
There, its statement is phrased using the `critical number' $\chi$, which for a set $X$ of elements
in a matroid $M \cong \PG(n-1,2)$ is defined by $\chi(X) = n - \omega_M(E(M)-X)$. 

\begin{lemma}\label{evenplanealpha}
  If $M \cong \PG(n-1,2)$ and $X \subseteq E(M)$ is a set 
  such that $|X \cap P|$ is even for every plane $P$ of $M$, 
  then $\omega_M(E(M)-X) \ge \ceil{\tfrac{n}{2}} - 1$. 
\end{lemma}

The next lemma follows from Lemma~\ref{complorsum}.
\begin{lemma}\label{compltrifreeomega}
  If $M \cong \PG(n-1,2)$ and $X \subseteq E(M)$ is a set
  with $\omega_M(X) \le 1$, then $\omega_M(E(M)-X) \ge C_{\ref{complorsum}}n^{1/7}$. 
\end{lemma}
\begin{proof}
  If $\wh{X} \subseteq \bF_2^n$ is the set of vectors corresponding to $X$,
  the fact that $\omega_M(X) \le 1$ implies that 
  $\wh{X} + \wh{X} \subseteq \bF_2^n - \wh{X}$, so $\alpha(\wh{X}) \ge \omega(\wh{X}+\wh{X})$. 
  Lemma~\ref{complorsum} therefore gives that 
  $\omega_M(E(M)-X) = \alpha(\wh{X}) \ge C_{\ref{complorsum}}n^{1/7}$. 
\end{proof}

For the next two results, let $\wh{c}$ be a 
$2$-colouring of the binary projective plane $\PG(2,2)$
in which the colour class of $1$ is a three-element independent set. 
Viewed as a binary matroid rather than a colouring, 
this is the \emph{claw}. The main theorem of [\ref{nn}], 
which we phrase here in terms of colourings, precisely describes
the structure of claw-free binary matroids. 

\begin{theorem}\label{clawfreestructure}
  Let $c$ be a $\wh{c}$-free $2$-colouring of $M \cong \PG(n-1,2)$ 
  with colour classes $X_1$ and $X_2$. 
  Then there are mutually skew flats $F_1, \dotsc, F_t$ of $M$ 
  such that $c = \bigotimes_i (c | F_i)$, and for each $i$, either
  \begin{enumerate}[(i)]
    \item\label{pgsum} $X_1 \cap F_i$ is the union of two disjoint flats of $F_i$ whose 
      union is spanning in $F_i$, 
    \item\label{compltrifree} $\omega_M(X_2 \cap F_i) \le 1$, or 
    \item\label{evenplane} for each plane $P$ of $F_i$, the set $X_1 \cap P$ has even size. 
  \end{enumerate} 
\end{theorem}

We can now prove Theorem~\ref{mainclaw}, which we rephrase in terms of two-colourings. 

\begin{theorem}\label{claweh}
  There exists $C > 0$ such that, 
  if $c$ is a $\wh{c}$-free two-colouring of $M \cong \PG(n-1,2)$, 
  then $M$ has a monochromatic flat of rank at least $C n^{1/7}$.
\end{theorem}
\begin{proof}
  Let $C' = \min(C_{\ref{complorsum}}, \tfrac{1}{2})$; we show that $C = C'/2^{1/7}$ 
  satisfies the theorem.
  Let $F_1, \dotsc, F_t$ be flats given by Lemma~\ref{clawfreestructure}. 
  Let $(I_1, I_2, I_3)$ be the partition of $\{1, \dotsc, t\}$ 
  corresponding to which of (\ref{pgsum}), (\ref{compltrifree}), (\ref{evenplane})
  holds for $F_i$. 
  For each $i \in I_1$, we clearly have 
  $\omega_M(F_i \cap X_1) \ge \tfrac{1}{2}r_M(F_i) \ge C' (r_M(F_i))^{1/7}$. 

  For each $i \in I_2$, applying Lemma~\ref{compltrifreeomega} to $F_i \cap X_1$
  in the matroid $M | F_i$ gives $\omega_M(F_i \cap X_1) \ge C' (r_M(F_i))^{1/7}$. 
  It follows from Lemma~\ref{ljomega} that 
  \[ \omega_M(X_1) \ge \sum_{i \in I_1 \cup I_2} C' (r_M(F_i))^{1/7} 
    \ge C' \left(\sum_{i \in I_1 \cup I_2}r_M(F_i)\right)^{1/7}.\] 

  For each $i \in I_3$, Lemma~\ref{evenplanealpha} gives
  $\omega_M(F_i \cap X_2) \ge \tfrac{1}{2}r_M(X_i),$ so  
  \[ \omega_M(X_2) \ge \sum_{i \in I_3} \tfrac{1}{2} r_M(F_i)
    = \tfrac{1}{2} \sum_{i \in I_3}r_M(F_i).\] 
  Since the $F_i$ are skew with spanning union, we have $\sum_{i} r_M(F_i) = n$, 
  so either $\sum_{i \in I_1 \cup I_2} r_M(F_i) \ge \tfrac{n}{2}$ or  
  $\sum_{i \in I_3} r_M(F_i) \ge \tfrac{n}{2}$.
  The first implies that $\omega_M(X_1) \ge (C'2^{-1/7})n^{1/7} = C n^{1/7}$ 
  and the second gives $\omega_M(X_2) \ge \tfrac{n}{2} \ge C n^{1/7}$, as required.
\end{proof}

\section{Infinite Fields}\label{infsection}

Matroids are traditionally defined to have finite ground sets, 
but this is certainly not necessary. 
Infinite matroids of finite rank can easily be defined from the usual matroid rank axioms by simply 
dropping the requirement that the ground set is finite 
while maintaining the stipulation that rank is integer-valued. 
With the notable exception of matroid duality,
most matroid theory works exactly as expected for these objects -- 
see [\ref{bdkpw}] for a discussion. 
It is natural to ask whether Theorem~\ref{mainstructuresimple} holds in this expanded setting, 
since it would then apply to colourings of finite-dimensional projective 
space over fields such as $\mathbb{Q}, \mathbb{R}$ and $\mathbb{C}$. 
Indeed, every component of the proof of Theorem~\ref{mainstructuresimple} 
easily adapts to  finite-rank infinite matroids, 
except Lemma~\ref{rankthree}, whose proof involves counting points in lines. 

Unfortunately, this issue is fundamental,
because the generalization of Theorem~\ref{mainstructuresimple} to projective 
geometries over infinite fields turns out to be false. 
We show via an algebraic construction that for many infinite fields $\bF$
(including $\mathbb{Q}, \mathbb{R}$ and $\mathbb{C}$), 
there is a rainbow-triangle-free $n$-colouring of $\PG(n-1,\bF)$
that cannot be decomposed via lift-joins. 
The construction is a higher-dimensional version of one discovered for projective planes 
by Carter and Vogt [\ref{cartervogt}]
and, independently, Hales and Straus [\ref{hs82}]. 

To explain the construction, 
we need some algebra. A \emph{non-Archimedean valuation} on a field $\bF$
is a function $v : \bF \to \bR \cup \{\infty\}$ so that $v^{-1}(\infty) = \{0\}$, 
and $v(ab) = v(a) + v(b)$ and $v(a+b) \ge \min(v(a), v(b))$ for all $a, b \in \bF$. 
Note that $v(a^{-1}) = - v(a)$ for all $a$ and therefore that $v(1) = v(-1) = 0$;
if there is some $a \ne 0$ for which $v(a) \ne 0$, then $v$ is \emph{nontrivial}.

It is routine to show that finite fields do not admit nontrivial non-Archimedean valuations, 
and it is known that the same is true for any algebraic extension of a finite field. 
Many other fields do have such valuations, however. For example, if $p$ is a prime, 
then the $p$-adic valuation on $\bQ$, defined by $v(p^t \tfrac{r}{s}) = t$ 
for all $r,s$ relatively prime to $p$, is non-Archimedean. 
Nontrivial valuations on $\bR$ and $\bC$ are harder to construct, but their existence 
follows from the axiom of choice. 

Given a non-Archimedean valuation $v$ on a field $\bF$, an integer $n \ge 1$, 
and a vector $(u_1, \dotsc, u_n) \in \bF^n$, 
write $c_v(u)$ for the smallest index $s$ such that $v(u_s) = \min_i v(u_i)$; 
this is an $n$-colouring of $\bF^n \del \{0\}$.

It is easy to check using the definition of a non-Archimedean valuation
that $c_v(\lambda u) = c_v(u)$ for all $\lambda \in \bF^{\times}$, 
so $c_v$ can also be viewed as an $n$-colouring defined on the projective space $\PG(n-1,\bF)$
whose points are the one-dimensional subspaces of $\bF^n$. 

(For $p$-adic valuations on the field of rationals, we can view $c_v$ very concretely:
to colour a point $x = (r_1, \dotsc, r_n)$, simply scale the $r_i$ to be integers with no common factor, 
and colour $x$ with the first index whose entry is not a multiple of $p$.)

For the remainder of the section, write $\PG(n-1,\bF)$ for the projective geometry whose points
are explicitly the $1$-dimensional subspaces of $\bF_q^n$, 
and let $[x]$ be the element of $\PG(n-1,\bF)$ corresponding to the vector $x \in \bF^n \del \{0\}$. 

We first show that the colouring $c_v$ is rainbow-triangle-free -- in fact
it is just as easy to show the equivalent statement that it has no rainbow circuit. 
(Showing this for all circuits has the tiny advantage of not exploiting
the equivalence in Lemma~\ref{rainbowcircuit}, 
which does hold for finite-rank infinite matroids, 
but we technically did not prove in this setting.)

\begin{theorem}\label{nonarchrainbow}
    Let $n \ge 1$, and let $v$ be a non-Archimedean valuation on a field $\bF$.
    Then every $c_v$-rainbow set in $\PG(n-1,\bF)$ is independent. 
\end{theorem}
\begin{proof}
    Let $I = \{x^1, \dotsc, x^t\}$ be a $c_v$-rainbow set, where 
    $c_v(x^1) < c_v(x^2) < \dotsc < c_v(x^t)$.
    Let $x^i = [(x^i_1, \dotsc, x^i_n)]$ and  $c_v(x^i) = k_i$ for each $i$.
    
    Write $|x|$ for $v(x)$. 
    Since $c_v(x^i_1, \dotsc, x^i_n) = k_i$ for each $i$, 
    we have $|x^i_{k_i}| \le |x^i_j|$ for all $j$, 
    and the inequality is strict if $j < k_i$. 

    To show that $I$ is independent, it suffices to show that the $t \times t$ matrix $A$
    defined by  $A_{ij} = x^i_{k_j}$ is nonsingular. 
    Suppose, therefore, that the determinant $|A|$ of $A$ is zero.
    Write $s(\pi) = \pm 1$ for the sign of a permutation $\pi \in S_n$.     For each $\pi \in S_n$, using $|s(\pi)| = 0$, we have 
    \[\left| s(\pi) \prod_{i=1}^t x^i_{k_{\pi(i)}} \right| = \sum_{i=1}^t |x^i_{k_{\pi(i)}}| \ge 
        \sum_{i=1}^t |x^i_{k_i}|.\]
    Moreover, if $\pi$ is not the identity permutation, there is some $s$ such that $\pi(s) < s$, 
    which implies that $k_{\pi(s)} < k_s$, so $|x^s_{k_{\pi(s)}}| > |x^s_{k_s}|$, 
    and the inequality above is strict. Therefore 
    \begin{align*}
        \sum_{i=1}^t |x^i_{k_i}| &< \min_{\pi \ne 1} \left|s(\pi)\prod_{i=1}^t x^i_{k_{\pi(i)}}\right| 
        \le \left|\sum_{\pi \ne 1} s(\pi)\prod_{i=1}^t x^i_{k_{\pi(i)}}\right| = 
            \left||A|- \prod_{i=1}^t x^i_{k_i}\right|.
    \end{align*}
    But $|A| = 0$ and $\left|- \prod_{i=1}^t x^i_{k_i}\right| = \sum_{i=1}^t |x^i_{k_i}|$, 
    giving a contradiction. 
\end{proof}

If $F$ is a nontrivial decomposer of a coloured projective geometry $(G; c)$ 
with $|c| \ge 2$, there clearly exist $x \in F$ and $y \ne F$ with $c(x) \ne c(y)$, 
and $x$ is the only point of the line $\cl_M(\{x,y\})$ with colour $c(x)$. 

The next theorem, on the other hand, 
shows that colourings $c_v$ arising from nontrivial non-Archimedean
valuations contain no lines that intersect a colour class in exactly one point, 
and hence they cannot be decomposed via lift-joins. 
This in turn implies that the statement of Theorem~\ref{mainstructuresimple} 
must fail for projective geometries over general infinite fields, 
even when it is adjusted to allow for infinite matroids. 

\begin{theorem}
    Let $n \ge 1$, and $v$ be a nontrivial non-Archimedean valuation on a field $\bF$.
    Then no line $L$ of $\PG(n-1,\bF)$ intersects a colour class of $c_v$ in exactly one point. 
\end{theorem}
\begin{proof}
    Suppose that $x = [(x_1, \dotsc, x_n)]$ is the unique point of a line $L$ 
    having colour $k \in [n]$. 
    Since $L$ contains no linearly independent triple, it contains no rainbow triple
    by Theorem~\ref{nonarchrainbow}, 
    so there is some colour $\ell \ne k$ so that $c(L-\{x\}) = \{k\}$. 
    
    Since $v$ is nontrivial, there exists $d \in \bF$ with $v(d) < 0$. 
    Define hyperplanes $W_1 = \{a \in \bF^n : a_k = 0\}$ 
    and $W_2 = \{a \in \bF^n : a_\ell = d a_k\}$ of $\bF^n$, 
    and let $H_1, H_2$ be the corresponding hyperplanes of $\PG(n-1,\bF)$. 

    For each $a \in W_1$, we have $a_k = 0$ and so $v(a_k) = \infty$ and hence $c_v(a) \ne k$ by the definition of $c_v$. 
    Similarly, for each $a \in W_2$, we have $v(a_{\ell}) = v(d) + v(a_k) < v(a_k)$, 
    so $c_v(a) \ne k$. But since $L$ is a line, there are points $f_1 \in L \cap H_1$ 
    and $f_2 \in L \cap H_2$. As just observed, we have $c_v(f_1) \ne k$ and $c_v(f_2) \ne k$; 
    since $f_1,f_2 \in L$ we must have $f_1 = f_2 = x$. 

    Therefore $x \in H_1 \cap H_2$, which by the definitions of $H_1$ and $H_2$ 
    implies that $x_k = x_{\ell} = 0$. 
    This contradicts the fact that $c_v(x) = k$.
\end{proof}

On the other hand, algebraic extensions of finite fields do not have nontrivial non-Archimedean valuations, 
and Theorem 1 of [\ref{hs82}] will imply Lemma~\ref{rankthree} for projective planes over all such fields. 
Using this fact, one could routinely extend 
Theorem~\ref{mainstructuresimple} to apply to all finite-dimensional coloured projective geometries 
over algebraic extensions of finite fields. 

\subsection*{Acknowledgements}
This project was initiated at the Workshop on (Mostly) Matroids held at the Institute for Basic Science in Daejeon, South Korea in 2024.  We thank Rutger Campbell and the other
organizers of the workshop for providing an extremely stimulating research environment. We also thank Denys Bulavka for helpful conversations, and Sean McGuinness for alerting 
us to the related work in [\ref{bs}]. Tony Huynh is supported by the Institute for Basic Science (IBS-R029-C1).

\section*{References}
\newcounter{refs}
\begin{list}{[\arabic{refs}]}
{\usecounter{refs}\setlength{\leftmargin}{10mm}\setlength{\itemsep}{0mm}}

\item\label{asw}
M. Axenovich, R. Snyder, L. Weber,
The Erdős-Hajnal conjecture for three colors and triangles,
Discrete Math. 345 (2022), Issue 5. 

\item\label{bs}
K. B\'{e}rczi, T. Schwarcz,
Rainbow and monochromatic circuits and cocircuits in binary matroids,
Discrete Math. 345 (2022), Issue 6, 112830

\item\label{bb}
R. C. Bose, R. C. Burton, 
A characterization of flat spaces in a finite geometry and the uniqueness of the Hamming and the MacDonald codes, 
J. Combin. Theory 1 (1966), 96--104. 

\item\label{bdm}
T. Boothby, M. DeVos, A. Montejano, 
A new proof of Kemperman's theorem,
Integers 15 (2015). 

\item\label{bdkpw}
H, Bruhn, R, Diestel, M. Kriesell, R. Pendavingh, P. Wollan,
Axioms for infinite matroids,
Adv. Math. 239 (2013), 18--46.

\item\label{bkknp}
M. Bonamy, F. Kardo\v{s}, T. Kelly, P. Nelson, L. Postle,
The structure of binary matroids with no induced claw or Fano plane restriction,
Advances in Combinatorics, 2019:1, 17pp.

\item\label{cartervogt}
D. S. Carter, A. Vogt,
Collinearity-preserving functions between Desarguesian planes,
Mem. Am. Math. Soc. 235:98 (1980).

\item\label{chudnovsky}
M. Chudnovsky, 
The Erdős–Hajnal Conjecture—A Survey, 
J. Graph Theory (2014), 75: 178-190.

\item \label{eh77}
P. Erdős, A. Hajnal,
On spanned subgraphs of graphs,
Graphentheorie und Ihre Anwendungen (Oberhof, 1977).

\item \label{eh89}
P. Erdős, A. Hajnal,
Ramsey-type theorems,
Discrete Applied Mathematics 25.1-2 (1989): 37-52.

\item\label{fgp}
J. Fox, A. Grinshpun, J. Pach, The Erdős-Hajnal conjecture for rainbow triangles, 
J. Combin. Theory Ser. B 111 (2015), 75--125.

\item\label{gallai}
T. Gallai, 
Transitiv orientierbare Graphen, 
Acta Math. Acad. Sci. Hungar 18 (1967), 25-66.
English translation in [\ref{gallaitrans}].

\item\label{gr}
R. L. Graham, K. Leeb, B. L. Rothschild,
Ramsey's theorem for a class of categories,
Adv. Math. 8 (1972), 417--433.

\item\label{hs82}
A. W. Hales, E. G. Straus, 
Projective Colorings,
Pacific. J. Math. 99 (1982), 31--43.

\item \label{hp}
Z. Hunter, C. Pohoata,
Some remarks on off-diagonal Ramsey numbers for vector spaces over $\bF_2$,
arXiv:2309.02424v1 [math.CO]

\item\label{km}
Z. Kelley, R. Meka, 
Strong Bounds for 3-Progressions,
2023 IEEE 64th Annual Symposium on Foundations of Computer Science (FOCS), 
Santa Cruz, CA, USA, 2023, pp. 933--973. 

\item\label{kemperman}
J. H. B. Kemperman, 
On small sumsets in an abelian group, 
Acta Math. 103 (1960), 63--88. 

\item\label{gallaitrans}
F. Maffray, M. Preissmann, 
A translation of Gallai’s paper: ‘Transitiv Orientierbare Graphen’,
In: Perfect Graphs (J. L. Ramirez-Alfonsin and B. A. Reed, Eds.), Wiley, New York, 2001, pp.
2566.

\item\label{mm}
C. Merino, J. J. Montellano-Ballesteros, 
Some heterochromatic theorems for matroids,
Discrete Math. 341 (2018), 2694--2699

\item\label{mo}
M. Mizell, J. G. Oxley, 
Matroids Arising From Nested Sequences of Flats In Projective And Affine Geometries,
Electron. J. Comb. 31 (2023), P2.48

\item\label{nn}
P. Nelson, K. Nomoto,
The structure of claw-free binary matroids,
J. Combin. Theory, Ser. B 150 (2021), 76--118.

\item \label{oxley}
J. G. Oxley, 
Matroid Theory,
Oxford University Press, New York (2011).

\end{list}

\end{document}

%% file: preamble.tex
\usepackage{amsmath,amssymb,enumerate,mathtools,dsfont}
\newtheorem{theorem}{Theorem}[section]
\newtheorem{claim}{}[theorem]
\newtheorem{lemma}[theorem]{Lemma}

\newtheorem{corollary}[theorem]{Corollary}
\newtheorem{conjecture}[theorem]{Conjecture}
\theoremstyle{definition}
\newtheorem{definition}[theorem]{Definition}

\newcommand{\bF}{\mathbb F}
\newcommand{\bR}{\mathbb R}
\newcommand{\bC}{\mathbb C}

\newcommand{\bN}{\mathbb N}
\newcommand{\bQ}{\mathbb Q}

\newcommand{\cL}{\mathcal{L}}

\newcommand{\cP}{\mathcal{P}}

\newcommand{\erdos}{Erd\H{o}s}
\newcommand{\es}{\varnothing}
\newcommand{\sse}{\subseteq}

\DeclareMathOperator{\cl}{cl}

\DeclareMathOperator{\PG}{PG}

\DeclareMathOperator{\GF}{GF}

\DeclareMathOperator{\codim}{codim}

\newcommand{\floor}[1]{\left\lfloor #1 \right\rfloor}
\newcommand{\ceil}[1]{\left\lceil #1 \right\rceil}

\newcommand{\del}{ \backslash  }
\newcommand{\con}{/}
\newcommand{\wh}{\widehat}

\numberwithin{subcase}{case}
\numberwithin{subsubcase}{subcase}
\newenvironment{subproof}[1][\proofname]{%
  \begin{proof}[Subproof:]%
}{%
  \end{proof}%
}